\documentclass[a4paper,11pt]{article}
\usepackage{MyStyle}

\usepackage{hyperref}
\usepackage{algorithm}
\floatname{algorithm}{Method}
\usepackage{algorithmic}

\allowdisplaybreaks

\newtheorem{theorem}{Theorem}[section]

\newtheorem{proposition}[theorem]{Proposition}
\newtheorem{lemma}[theorem]{Lemma}
\newtheorem{remark}[theorem]{Remark}
\newtheorem*{remark*}{Remark}

\newtheorem*{remarks*}{Remarks}

\newtheorem{ass}[theorem]{Assumption}
\newtheorem*{notation*}{Notation}

\newtheorem*{ex*}{Example}

\newtheorem*{exs*}{Examples}

\newtheorem*{app*}{Application}
\newtheorem{conjecture*}{Conjecture}

\def\ts{\thinspace}

\newcommand{\invariant}{Q}
\newcommand{\invariantmatrix}{S}
\newcommand{\methodA}{A}
\newcommand{\methodB}{B}

\title{
Multirevolution integrators for differential equations with fast stochastic oscillations
}

\author{
Adrien Laurent\textsuperscript{1} and Gilles Vilmart\textsuperscript{1}
}

\begin{document}
\footnotetext[1]{
Universit\'e de Gen\`eve, Section de math\'ematiques, 2-4 rue du Li\`evre, CP 64, CH-1211 Gen\`eve 4, Switzerland. Adrien.Laurent@unige.ch, Gilles.Vilmart@unige.ch.}

\maketitle

\begin{abstract}
We introduce a new methodology based on the multirevolution idea for constructing integrators for stochastic differential equations in the situation where the fast oscillations themselves are driven by a Stratonovich noise.
Applications include in particular highly-oscillatory Kubo oscillators and spatial discretizations of the nonlinear Schrödinger equation with fast white noise dispersion.
We construct a method of weak order two with computational cost and accuracy both independent of the stiffness of the oscillations.
A geometric modification that conserves exactly quadratic invariants is also presented.

\smallskip

\end{abstract}


\section{Introduction}

This article aims at developing invariant-preserving integrators of second weak order that are robust with respect to the stiffness $\varepsilon$ both in accuracy and cost for the following class of highly-oscillatory $d$-dimensional SDEs driven by a one-dimensional Stratonovich noise
\begin{equation}
\label{equation:NLS}
dX(t)=\frac{1}{\sqrt{\varepsilon}}A X(t)\circ dW(t)+F(X(t)) dt,\ t>0,\ X(0)=X_0,
\end{equation}
where $W$ is a standard one-dimensional Wiener process, the function $F:\R^d\rightarrow \R^d$ is a smooth non-linear map, the stiff parameter $\varepsilon>0$ is fixed and assumed small, and $A\in \R^{d\times d}$ is a given matrix satisfying $e^A=\Id$ (equivalently $A$ is diagonalizable and has all its eigenvalues in $2i\pi\Z$).
In the deterministic setting, this last property yields that the solution $x(t)=\exp(\varepsilon At)x_0$ of $\frac{dx}{dt}=\varepsilon^{-1}A  x$ is $\varepsilon$-periodic.
For stochastic oscillations, it means that the solution $X(t)=\exp(\varepsilon^{-1/2}AW(t))X_0$ of $dX=\varepsilon^{-1/2}A  X \circ dW$ satisfies $X(T)=X(0)$ for a random time $T=\inf\{t>0, \left|\varepsilon^{-1/2}W(t)\right|=1\}$ of mean $\varepsilon$.
The class of SDEs \eqref{equation:NLS} includes in particular highly-oscillatory Kubo oscillators (see \cite{Cohen12otn})
\begin{equation}
\label{equation:Kubo_oscillator}
dX=\frac{2\pi}{\sqrt{\varepsilon}} \begin{pmatrix}
0&-1\\1&0
\end{pmatrix} X\circ dW+\begin{pmatrix}
0&-a\\a&0
\end{pmatrix} X dt,\ a\in\R,
\end{equation}
or equivalently, $dY=2i\pi\varepsilon^{-1/2} Y\circ dW+ia Y dt$ in the complex setting where $Y=X_1+i X_2$.

Applying standard SDE integrators to solve equation \eqref{equation:NLS} requires in general a time stepsize $h\leq \varepsilon$ to be accurate, which makes these methods dramatically expensive when $\varepsilon$ is small. The goal of this paper is to create robust numerical methods, i.e.\ts numerical integrators whose cost and accuracy do not deteriorate when $\varepsilon$ becomes small.
Several classes of methods have already been developed for highly-oscillatory SDEs with a deterministic fast oscillation (see for instance \cite{Cohen12cao,Vilmart14wso}), but not in the case where the stiff oscillatory part is applied to the noise itself.
To numerically face this challenge, we introduce in this paper a new methodology to develop robust methods of any high weak order to approximate the solution of equation \eqref{equation:NLS}. In particular, we propose a method of weak order two, and a geometric modification of this algorithm that preserves quadratic invariants.

Stochastic oscillations as defined in \eqref{equation:NLS} typically arise in fiber optics models (see \cite{Agrawal07nfo,Agrawal08aon,Garnier02sod}) with a spatial discretizations of the highly-oscillatory nonlinear Schr\"odinger equation (NLS) with white noise dispersion
\begin{equation}
\label{equation:NLS_WND_general}
du(t)=\frac{i}{\sqrt{\varepsilon}}\Delta u(t)\circ dW(t)+F(u(t)) dt,\ u(t=0)=u_0.
\end{equation}
As described for instance in \cite{Garnier02sod}, in the case $\varepsilon=1$, the NLS equation \eqref{equation:NLS_WND_general} with a cubic nonlinearity $F(u)=|u|^2u$ is a model in dimension $d=1$ describing the propagation of a signal in optical fibers where $x$ corresponds to the retarded time, while $t$ corresponds to the distance along the fiber. Taking into account the inevitable chromatic dispersion effects of the signal, modeled by a random centered stationary process $m$ with a coefficient $\nu>0$,
yields the following random PDE,
$$
\frac{\partial v}{\partial x}(x,t)=\nu i m(x)\frac{\partial^2 v}{\partial t^2}(x,t)+\nu^2 F(v(x,t)),\ v(x=0,t)=u_0(t).
$$
The perfect fiber would satisfy $m=0$, but in practice, engineers build fibers with a small varying dispersion coefficient.
To limit the pulse broadening induced by random dispersion, specialists use a wide range of dispersion management techniques (see for instance \cite{Garnier02sod} and references therein).
In \cite{Marty06oas,DeBouard10tns}, the authors show that if we denote $u^{\nu}(x,t)=v(x/\nu^2,t)$, then as $\nu$ tends to $0$ and under some ergodicity assumptions on $m$, $u^{\nu}$ converges to the solution $u$ of equation \eqref{equation:NLS_WND_general} with $\varepsilon=1$.
The non-stiff counterpart of equation \eqref{equation:NLS_WND_general}, i.e.\ts for $\varepsilon=1$, has also been studied theorically in \cite{Debussche11qns} for a particular nonlinearity.
The highly-oscillatory behaviour ($\varepsilon \ll 1$) appears naturally when observing the propagation in long time with a small nonlinearity (via the change of variable $t\leftarrow \varepsilon t$) or the propagation of a small initial data in an optical fiber with a polynomial nonlinearity (via the change of variable $u\leftarrow u/\varepsilon$).
A goal of this article is to develop efficient and cheap numerical methods that can model the propagation of pulses in this context, in order to observe some specific behaviors and, ultimately, to build enhanced fibers.
Models of the form \eqref{equation:NLS_WND_general} also appear in the recent work \cite{Faou18lwt} in the context of stochastic three-wave semi-linear systems.
We emphasize that there is a growing interest in the recent litterature for stochastic models involving a fast Stratonovitch noise in the context of ergodic stochastic dynamics. In \cite{Abdulle19act}, it is shown for a class of overdamped Langevin equations that adding an appropriate fast Stratnovitch noise permits to increase the convergence rate to equilibrium, while reducing the asymptotic variance at infinity. This suggests that new efficient samplers for the invariant distribution of Langevin type models in context of large dimensional molecular dynamics models could be developed.
We also mention the recent homogenization results on stochastic dynamics with fast Stratonovitch noises in \cite{Li18hoh} where our periodicity assumption is replaced by an ergodicity assumption on the fast component of the dynamics posed on manifolds.

Numerous possibilities exist for numerically integrating equations \eqref{equation:NLS} or \eqref{equation:NLS_WND_general}.
We highlight in particular the exponential integrators \cite{Cohen12otn,Erdogan18anc} for the SDE \eqref{equation:NLS}, and the exponential integrators \cite{Cohen17eif}, the Fourier split-step method \cite{Marty06oas} or the Crank-Nicholson scheme \cite{Belaouar15nao} for the SPDE \eqref{equation:NLS_WND_general}.
These methods have the advantage that they preserve the $L^2$ invariant of the equation (that is $\norme{u(t)}_{L^2}=\norme{u_0}_{L^2}$ for all $t\geq 0$) for a class of polynomial nonlinearities.
However they face a severe timestep restriction $h\leq \varepsilon$ when the stiff parameter $\varepsilon$ is small.
Even in the case of deterministic oscillations, there are restrictions in general, though some robust algorithms exist (see \cite{Cohen12cao} for instance).
The methods presented in this paper solve this issue of stepsize restriction.
The idea is to approximate the solution of equation \eqref{equation:NLS} at random times called revolution times because they correspond to complete revolutions of the oscillatory part $dX=\varepsilon^{-1/2}A  X \circ dW$. This is in the spirit of \cite{Hofmann00oao} which also approximates the solution of SDEs at random times.

The article is organized as follows.
Section \ref{section:multirevolution_integrators} is devoted to the presentation of the new integrators.
In Section \ref{section:construction_algorithm}, we build an asymptotic expansion of the solution of \eqref{equation:NLS} and evaluate it at revolution times to derive the new integrators and a limit model for equation \eqref{equation:NLS}.
Section \ref{section:convergence_theorem} is devoted to the weak convergence theorems and their proofs.
In Section \ref{section:numerical experiments}, we present numerical experiments to confirm our theoretical error estimates, and we apply the new methods to solve numerically the Schrödinger equation \eqref{equation:NLS_WND_general}.

\section{Multirevolution integrators for stochastic oscillators}
\label{section:multirevolution_integrators}

Initially created in \cite{Melendo97ana, Calvo04aco} in the context of celestial mechanics and later extended using geometric integration (see for instance \cite{Murua99ocf, Calvo07oem, Chartier14mrc}), multirevolution methods represent a class of numerical methods used for solving highly-oscillatory differential equations while reducing the cost of computation.
\begin{figure}[t]
	\begin{center}
		\includegraphics[scale=0.47]{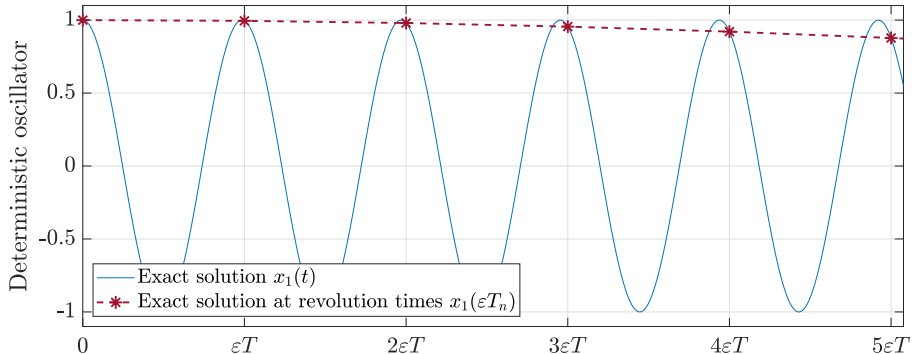}
		\caption{Exact solution evaluated at revolution times for the deterministic oscillator \eqref{equation:highly_oscillatory_ODE} with $F(y)=iy$ and $\varepsilon=10^{-1}$.}
		\label{figure:Comparison_revolution_time_deterministic}
	\end{center}
\end{figure}
\noindent
In particular, they can approximate the solution of highly-oscillatory ODEs of the following form at stroboscopic times $\varepsilon N T$, where $T=1$ is the period of $\frac{dx}{dt}=Ax$, and $N$ is an integer,
\begin{equation}
\label{equation:highly_oscillatory_ODE}
\frac{dx}{dt}=\frac{1}{\varepsilon}A x+F(x), \ x(0)=x_0.
\end{equation}
The solution $x$ of this equation at times $\varepsilon NT$ is a perturbation of identity, that is $x$ satifies $x(\varepsilon t)=x_0+\OO(\varepsilon t)$, thus the solution loses its highly-oscillatory feature when evaluated at stroboscopic times, as shown in Figure \ref{figure:Comparison_revolution_time_deterministic} for the first component of the solution of equation \eqref{equation:highly_oscillatory_ODE} with $F(x)=ix$ (respectively $F(y)=\begin{pmatrix}
0&-1\\1&0
\end{pmatrix}y$ in the real setting).
The idea of multirevolution is to approximate $x(\varepsilon N)$ with $N=\OO(\varepsilon^{-1})$ with a cost independent of $\varepsilon$.

\begin{figure}[t]
	\begin{center}
		\includegraphics[scale=0.47]{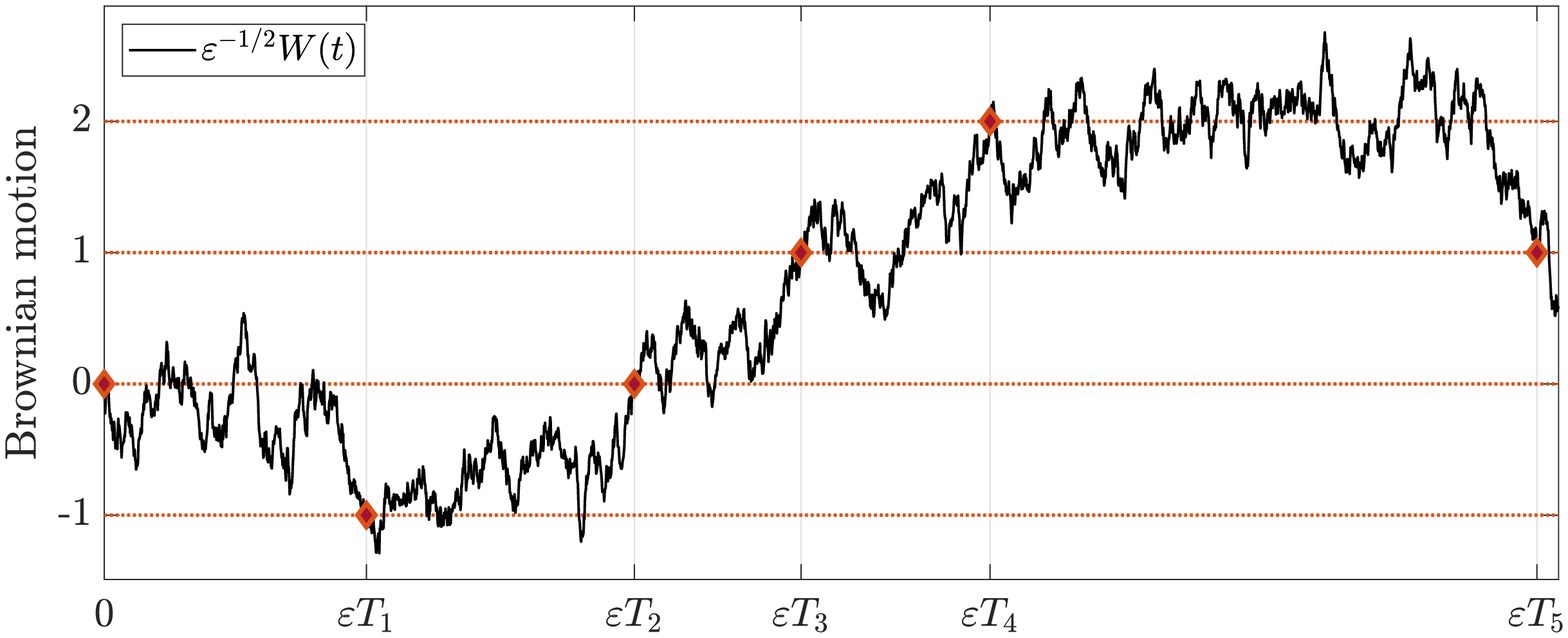}
		\includegraphics[scale=0.47]{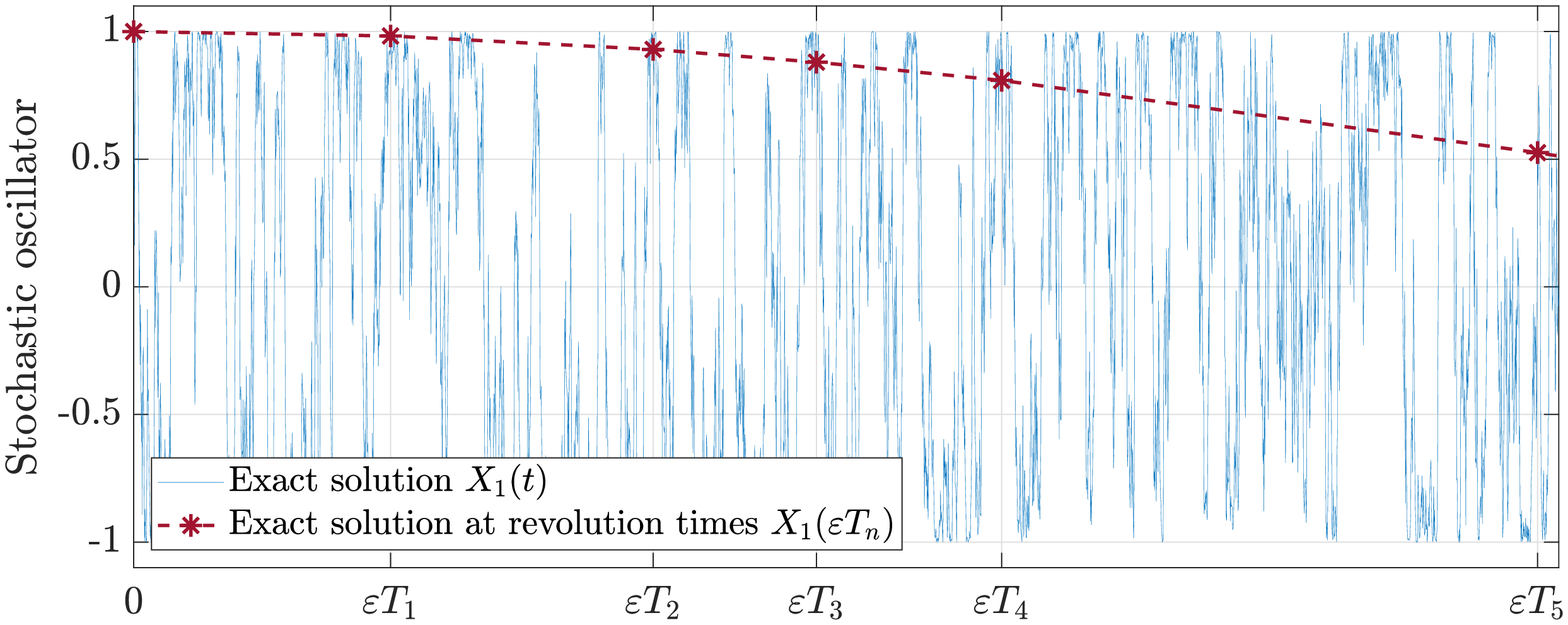}
		\caption{Revolution times \eqref{equation:definition_revolution_times} of a Brownian path (top) and exact solution evaluated at revolution times for the Kubo oscillator \eqref{equation:Kubo_oscillator} with $a=1$ and $\varepsilon=10^{-1}$ (bottom).}
		\label{figure:Comparison_revolution_time_stochastic}
	\end{center}
\end{figure}
For stochastic oscillations, the solution $X(t)=e^{\varepsilon^{-1/2}AW(t)}X_0$ of $dX=\varepsilon^{-1/2}AX\circ dW$ is not periodic, but satisfies $X(\varepsilon T_N)=X_0$ where the $T_N$ are random variable called revolution times and defined by
\begin{align}
\label{equation:definition_revolution_times}
T_0&=0,\\
T_{N+1}&=\inf\left\{t>T_N, \varepsilon^{-1/2}\abs{W(\varepsilon t)-W(\varepsilon T_N)}\geq 1\right\}, \ N=0,1,2,\ldots \nonumber
\end{align}
If $X$ is the solution of \eqref{equation:NLS}, we show in Section \ref{section:asymptotic_expansion} that $X$ evaluated at times $\varepsilon T_N$ is a perturbation of identity (in a strong and weak sense).
Figure \ref{figure:Comparison_revolution_time_stochastic} illustrates the definition of revolution times and shows the perturbation of identity property on the first component of a Kubo oscillator \eqref{equation:Kubo_oscillator} with $a=1$.
We highlight that the revolutions times $T_N$ can be simulated without simulating the exact path $W$. Also we emphasize that the proposed algorithms do not require to simulate $W$ thanks to the use of appropriate discrete random variables. This will be detailed in Section \ref{section:stiff_integrators}.

We show in Section \ref{section:limit_model} that the solution $X$ of \eqref{equation:NLS} evaluated at times $\varepsilon T_{t\varepsilon^{-1}}$ (when $t/\varepsilon\in\N$ is an integer) converges weakly when $\varepsilon\rightarrow 0$ to the solution $y_t$ of the deterministic ODE
$$\frac{d y_t}{dt}=\langle g^0\rangle (y_t), \ y_{0}=X_0,$$
where $g_\theta^0(y) = e^{-A\theta} F(e^{A\theta}y)$ and $\langle g^0\rangle :=\int_0^1 g_\theta^0 d\theta$. This ODE is exactly the same one as the asymptotic model for deterministic oscillators of the form \eqref{equation:highly_oscillatory_ODE}.
This asymptotic model naturally yields a weak order 1 deterministic integrator.
We propose the two following new multirevolution methods of second weak order for integrating equation \eqref{equation:NLS} at the revolution times $\varepsilon T_{Nm}$ for $m=0, 1, 2,\ldots$ with cost in $H=N\varepsilon=\OO(1)$ independent of $\varepsilon$.
Method \ref{algorithm:geometric_weak_order_2} is a geometric modification of Method \ref{algorithm:weak_order_2} to preserve quadratic invariants of the form $\invariant(y)=\frac{1}{2}y^T \invariantmatrix y$ where $ \invariantmatrix \in \R^{d\times d}$ is a given symmetric matrix.
Methods A and B involve a Fourier decomposition of the following functions
that are 1-periodic with respect to $\theta$, 
\begin{align}
\label{equation:def_g_theta}
g_\theta^0(y) &= e^{-A\theta} F(e^{A\theta}y)=\sum_{k\in\Z} c_k^0(y) e^{2i\pi k \theta}\\
g_\theta^1(y)(z) &= e^{-A\theta}F'(e^{A\theta}y)(e^{A\theta} z) =\sum_{p\in\Z} c_p^1(y)(z) e^{2i\pi p \theta}\nonumber
\end{align}
with respective Fourier coefficients $(c_k^0(y))_{k\in\Z}$ and $(c_p^1(y))_{p\in\Z}$.
The series appearing in \eqref{equation:def_g_theta} have an infinite number of terms in general. For a practical implementation of the new methods, we truncate these series up to an even number of modes $K_t$, while inducing an exponentially small error (see Remark~\ref{rem:trunc}).
For each timestep, we also introduce the bounded discrete random variables $(\widehat{\alpha}_k^N)_k$, and deterministic sequences $(\widehat{\beta}_{p,k}^N)_{p,k}$ and $(\widehat{\widetilde{\beta}}_{p,k}^N)_{p,k}$ that satisfy
\vspace{-18pt}
\begin{multicols}{2}
$$
\begin{array}{rl}
\E[\widehat{\alpha}_k^N]&=\left\{
\begin{array}{l}
1 \text{ if } k=0\\
0 \text{ else}
\end{array}
\right.\\
\widehat{\beta}_{p,k}^N&=\left\{
\begin{array}{l}
\frac{1}{2}+\frac{1}{3N} \text{ if } p=k=0\\
\frac{1}{2\pi^2k^2N} \text{ if } p=0,\:  k\neq 0\\
\frac{-1}{2\pi^2p^2N} \text{ if } p\neq 0,\:  k= 0\\
\frac{1}{2\pi^2p^2N} \text{ if } p+k= 0,\: p,k\neq 0\\
0 \text{ else}
\end{array}
\right.
\end{array}
\hskip-0.5cm
\begin{array}{rl}
\E[\widehat{\alpha}_p^N\widehat{\alpha}_k^N]&=\left\{
\begin{array}{l}
1+\frac{2}{3N} \text{ if } p=k=0\\
\frac{1}{\pi^2p^2N} \text{ if } p+k= 0,\: p,k\neq 0\\
0 \text{ else}
\end{array}
\right.\\
\widehat{\widetilde{\beta}}_{p,k}^N&=\left\{
\begin{array}{l}
\frac{1}{2\pi^2k^2N} \text{ if } p=0,\:  k\neq 0\\
\frac{-1}{2\pi^2p^2N} \text{ if } p\neq 0,\:  k= 0\\
0 \text{ else}
\end{array}
\right.\\
\end{array}
$$
\end{multicols}%
\noindent 
The definition and construction of these random variables is further discussed in Section \ref{section:probabilistic_properties} and Section \ref{section:stiff_integrators}.

\begin{algorithm}[H]
\renewcommand{\thealgorithm}{\methodA}
\caption{(Explicit integrator of weak order two in $H=N\varepsilon$ to approximate the solution of equation \eqref{equation:NLS} at times $\varepsilon T_{Nm}$ for $m=0, 1, 2,\ldots$)
}
\label{algorithm:weak_order_2}
\begin{algorithmic}
\STATE $Y_0=X_0$
\FOR{$m\geq 0$}
\STATE 
\begin{equation}
\label{equation:order_2_scheme}
Y_{m+1} = Y_m
+H \sum_{k=-K_t /2}^{K_t /2-1} c_k^0(Y_m) \widehat{\alpha}_k^N
+H^2 \sum_{p,k=-K_t /2}^{K_t /2-1} c_p^1(Y_m) (c_k^0(Y_m)) \widehat{\beta}_{p,k}^N
\end{equation}
\ENDFOR
\end{algorithmic}
\end{algorithm}

\begin{algorithm}
\renewcommand{\thealgorithm}{\methodB}
\caption{(Geometric integrator of weak order two in $H=N\varepsilon$ to approximate the solution of equation \eqref{equation:NLS} at times $\varepsilon T_{Nm}$ for $m=0, 1, 2,\ldots$ while preserving quadratic invariants)
}
\label{algorithm:geometric_weak_order_2}
\begin{algorithmic}
\STATE $Y_0=X_0$
\FOR{$m\geq 0$}
\STATE
\vspace{-5 mm}
\begin{align}
\label{equation:geometric_order_2_scheme}
Y_{m+1} &= Y_m
+H \sum_{k=-K_t /2}^{K_t /2-1} c_k^0\left(\frac{Y_m+Y_{m+1}}{2}\right) \widehat{\alpha}_k^N\\
&+H^2 \sum_{p,k=-K_t /2}^{K_t /2-1} c_p^1\left(\frac{Y_m+Y_{m+1}}{2}\right)\left(c_k^0\left(\frac{Y_m+Y_{m+1}}{2}\right)\right) \widehat{\widetilde{\beta}}_{p,k}^N \nonumber
\end{align}
\ENDFOR
\end{algorithmic}
\end{algorithm}

\begin{remark}
One could apply a Newton iteration to solve the implicit equation \eqref{equation:geometric_order_2_scheme} in Method \ref{algorithm:geometric_weak_order_2}.
However a few fixed point iterations are sufficient (see discussion in \cite[Chap.\ts VIII]{Hairer06gni} for non-stiff implicit methods).
Indeed, since the Lipschitz constant of the iterated map has size $\OO(H)$, the convergence rate of the fixed point iterations is independent of the smallness of $\varepsilon$. 
\end{remark}

\begin{remark}
We observe that $\widehat{\beta}_{p,k}^N$ and $\widehat{\widetilde{\beta}}_{p,k}^N$ are always zero except when $p=0$, $k=0$ or $p+ k=0$. Thus the computational cost of one step of Methods \ref{algorithm:weak_order_2} and \ref{algorithm:geometric_weak_order_2} grows linearly in the number of modes in \eqref{equation:def_g_theta}.
\end{remark}

\section{Analysis and asymptotic expansion of the exact solution}
\label{section:construction_algorithm}

In this section, we first obtain a local expansion of the solution of \eqref{equation:NLS} and then evaluate it at particular random times to deal with the highly-oscillatory patterns of the exact solution. Finally we derive from this expansion an asymptotic limit for equation \eqref{equation:NLS} when $\varepsilon\rightarrow 0$.

\subsection{Asymptotic expansion of the exact solution}
\label{section:asymptotic_expansion}

Instead of studying directly equation \eqref{equation:NLS}, we apply the change of variable $t\leftarrow\varepsilon^{-1} t$ to obtain the following equation, whose solution satisfies $Y(t)=X(\varepsilon t)$ with $X$ solution of \eqref{equation:NLS},
\begin{equation}
\label{equation:NLS_rescaled}
dY(t)=A  Y(t) \circ d\widetilde{W}(t)+\varepsilon F(Y(t)) dt, \ Y(0)=X_0,
\end{equation}
where we denote for simplicity the Brownian motion $\widetilde{W}(t)=\varepsilon^{-1/2} W(\varepsilon t)$ again by $W$.
We introduce the following assumption which guaranties in particular global existence and uniqueness of the solution.
\begin{ass}
\label{assumption:F_Lipschitz}
The function $F$ is globally Lipschitz continuous and lies in $\CC^3_P$, i.e.\ts there exists constants $L$, $C$, $K>0$ such that for all $y$, $y_1$, $y_2\in\R^d$
\begin{equation}
\label{equation:regularity_F}
\abs{F(y_1)-F(y_2)}\leq L\abs{y_1-y_2} \qquad \abs{F^{(i)}(y)}\leq C (1+\abs{y}^K),\ i\in \{0,1,2,3\}.
\end{equation}
Also the initial condition $X_0$ has bounded moments, that is $\E[\abs{X_0}^p]<\infty$ for $p\geq 0$.
\end{ass}
Therefore we denote $\varphi_{\varepsilon,t}(X_0)=Y(t)$ the solution of equation \eqref{equation:NLS_rescaled} and focus in the rest of the paper on the approximation of $\varphi_{\varepsilon,t}(y)$ at times $t=\OO(\varepsilon^{-1})$.
The variation of constants formula yields
\begin{equation}
\label{equation:variation_constants}
\varphi_{\varepsilon,t}(y)=e^{AW(t)}y+\varepsilon \int_0^t e^{A(W(t)-W(s))} F(\varphi_{\varepsilon,s}(y)) ds.
\end{equation}
We deduce the following regularity properties.
\begin{lemma}
\label{lemma:regularity_varphi}
Under Assumption \ref{assumption:F_Lipschitz}, the following estimates hold for all $y$, $y_1$, $y_2\in\R^d$, where $C$ and $K$ are independent of $\varepsilon$ and $t$,
\begin{enumerate}
\item $\abs{\varphi_{\varepsilon,t}(y_1)-\varphi_{\varepsilon,t}(y_2)}\leq C\abs{y_1-y_2}e^{C\varepsilon t}$,
\item $\abs{\varphi_{\varepsilon,t}(y)}\leq C(1+\abs{y})e^{C\varepsilon t}$,
\item $\varphi_{\varepsilon,t}(y)$ is $\CC^3$ in $y$ and $\abs{\varphi_{\varepsilon,t}^{(i)}(y)}\leq C(\varepsilon t)^{i-1}(1+\abs{y}^K)e^{C\varepsilon t}$ for $i\in {1,2,3}$.
\end{enumerate}
\end{lemma}

The proof is postponed to the appendices. It mainly relies on the Gronwall theorem and the boundedness of the one-periodic function $\theta \mapsto e^{\theta A}$.
Using a local expansion of the solution of \eqref{equation:NLS_rescaled} in $\varepsilon$, we define the following first and second order approximations of $\varphi_{\varepsilon,t}(y)$,
\begin{align}
\label{equation:strong_Taylor_development}
\psi_{\varepsilon,t}^1(y)&=e^{AW(t)}y+\varepsilon e^{AW(t)} \int_0^{t} e^{-AW(s)} F(e^{AW(s)}y) ds\\
\psi_{\varepsilon,t}^2(y)&=\psi_{\varepsilon,t}^1(y) + \varepsilon^2 e^{AW(t)} \int_0^{t} e^{-AW(s)}  \nonumber \\
&\cdot F'(e^{AW(s)}y)\left( e^{AW(s)} \int_0^{s} e^{-AW(r)} F(e^{AW(r)}y) dr \right)ds. \nonumber
\end{align}

\begin{proposition}[Local expansion]
\label{proposition:strong_Taylor_development}
Under Assumption \ref{assumption:F_Lipschitz}, for all $y\in\R^d$, $j\in \{1,2\}$ and $t\geq 0$, there exists $C$ and $K$ two positive constants independent of $\varepsilon$ and $t$ such that
$$\abs{\varphi_{\varepsilon,t}(y)-\psi_{\varepsilon,t}^j(y)}\leq C(1+\abs{y}^K)e^{C\varepsilon t}(\varepsilon t)^{j+1}.$$
\end{proposition}

The functions $\psi_{\varepsilon,t}^j$ satisfy the following straightforward inequalities proved with similar arguments as for Lemma \ref{lemma:regularity_varphi}.
\begin{lemma}
With the assumptions and notations of Proposition \ref{proposition:strong_Taylor_development}, the following estimates hold for all $y\in\R^d$, where $C$ and $K$ are independent of $\varepsilon$ and $t$,
\begin{align}
\label{equation:Taylor_bound_1}
\abs{\psi_{\varepsilon,t}^1(y)}&\leq C(1+\abs{y})e^{C\varepsilon t},\\
\label{equation:Taylor_bound_2}
\abs{\psi_{\varepsilon,t}^2(y)}&\leq C(1+\abs{y}^K)e^{C\varepsilon t},\\
\label{equation:Taylor_bound_3}
\abs{\psi_{\varepsilon,t}^2(y)-e^{AW(t)}y}&\leq C(1+\abs{y}^K)(\varepsilon t)e^{C\varepsilon t}.
\end{align}
\end{lemma}

\begin{proof}[Proof of Proposition \ref{proposition:strong_Taylor_development}]
Using Assumption \ref{assumption:F_Lipschitz}, we get
$$\abs{\varphi_{\varepsilon,t}(y)-\psi_{\varepsilon,t}^1(y)}\leq L\varepsilon \int_0^{t} \abs{\varphi_{\varepsilon,s}(y)-e^{AW(s)}y} ds.$$
Then Lemma \ref{lemma:regularity_varphi} yields
\begin{align*}
\abs{\varphi_{\varepsilon,s}(y)-e^{AW(s)}y}&\leq C\varepsilon \int_0^{s} \abs{F(\varphi_{\varepsilon,r}(y))} dr
\leq C\varepsilon \int_0^{s} (1+\abs{\varphi_{\varepsilon,r}(y)}) dr\\
&\leq C\varepsilon \int_0^{s} (1+C(1+\abs{y})e^{C\varepsilon r}) dr
\leq C(1+\abs{y})e^{C\varepsilon s}(\varepsilon s).
\end{align*}
We deduce $\abs{\varphi_{\varepsilon,t}(y)-\psi_{\varepsilon,t}^1(y)}\leq C(1+\abs{y})e^{C\varepsilon t}(\varepsilon t)^2.$

For $j=2$, we first denote
$$\widetilde{\psi}_{\varepsilon,t}^2(y)=e^{AW(t)}y+\varepsilon e^{AW(t)} \int_0^{t} e^{-AW(s)} F(\psi_{\varepsilon,s}^1(y)) ds.$$
With the same arguments we used for $j=1$ and inequality \eqref{equation:Taylor_bound_1}, we have
$$\abs{\varphi_{\varepsilon,t}(y)-\widetilde{\psi}_{\varepsilon,t}^2(y)}\leq C(1+\abs{y}^K)e^{C\varepsilon t}(\varepsilon t)^3.$$
It is sufficient to prove that $\abs{\psi_{\varepsilon,t}^2(y)-\widetilde{\psi}_{\varepsilon,t}^2(y)}\leq C(1+\abs{y}^K)e^{C\varepsilon t}(\varepsilon t)^3$.
A Taylor expansion of $F(\psi_{\varepsilon,s}^1(y))$ in $\varepsilon$ gives
$$F(\psi_{\varepsilon,s}^1(y))=F(e^{AW(s)}y)+\varepsilon F'(e^{AW(s)}y)\left( e^{AW(s)} \int_0^{s} e^{-AW(r)} F(e^{AW(r)}y) dr \right) +R_{\varepsilon,s}.$$
The remainder $R_{\varepsilon,s}$ satisfies
$$\abs{R_{\varepsilon,s}}\leq C \varepsilon^2 \sup_{x \in [e^{AW(s)}y , \psi_{\varepsilon,s}^1(y)]} \norme{F''(x)} \left| e^{AW(s)} \int_0^{s} e^{-AW(r)} F(e^{AW(r)}y) dr \right|^2.$$
Then, using the polynomial growth of $F''$ and inequality \eqref{equation:Taylor_bound_1}, we get
$$\abs{R_{\varepsilon,s}}\leq C(1+\abs{e^{AW(s)}y}^K+\abs{\psi_{\varepsilon,s}^1(y)}^K) (\varepsilon s)^2 e^{C\varepsilon s}
\leq C(1+\abs{y}^K) (\varepsilon s)^2 e^{C\varepsilon s}.$$
Hence the result.
\end{proof}


We shall prove in Section \ref{section:limit_model} that the function $\psi^2_{\varepsilon,t}$ in \eqref{equation:strong_Taylor_development} evaluated at the revolution times $T_N$ (defined in \eqref{equation:definition_revolution_times}) yields a strong order 2 approximation in $H=\varepsilon N$.

\begin{remark}
If we replace the Brownian motion $W$ in \eqref{equation:strong_Taylor_development} by a piecewise linear function $W_\tau$ defined by
\begin{equation}
\label{equation:W_piecewise_linear}
W_\tau=\left(1-\frac{t}{\tau}+i\right)W_i+\left(\frac{t}{\tau}-i\right)W_{i+1} \text{ for } i \tau \leq t \leq (i+1)\tau,
\end{equation}
where $W_0=0$ and $W_{i+1}=W{i}+\sqrt{\tau}\xi_i$ with $(\xi_i)_i$ a family of independent standard Gaussian random variables, then it can be shown that we obtain an integrator of strong order two in $\varepsilon t$. However the cost of a standard method computing an approximation of the integrals of equation \eqref{equation:strong_Taylor_development} by replacing $W$ with $W_\tau$ is in $\OO(t^2/\tau^2)$, which makes this method tremendously expensive for $t=\OO(\varepsilon^{-1})$.
This is why we develop in Section \ref{section:stiff_integrators} weak integrators based on a weak approximation of equation \eqref{equation:strong_Taylor_development} with a cost independent of $t$.
We shall replace stochastic integrals by appropriate discrete random variables in order not to simulate any expensive Brownian path $W$.
\end{remark}

\subsection{Properties of the revolution times}
\label{section:probabilistic_properties}

In this section, we study some properties linked to the revolution times $T_N$ that will be useful for the analysis.
\begin{proposition}
\label{proposition:properties_T1}
The revolution times $T_N$ defined in \eqref{equation:definition_revolution_times} are positive and finite almost surely. Their differences $(T_{N+1}-T_N)_{N\geq 0}$ are independent identically distributed random variables (same law as $T_1$). The Laplace transform $\E[e^{zT_1}]$ of $T_1$ exists and is analytic for $\Real(z)<\frac{\pi^2}{8}$. In addition, for $x\in\left[0,\frac{\pi^2}{8}\right[$, $\E[e^{xT_1}]=\frac{1}{\cos(\sqrt{2x})}$.
The variable $T_1$ has bounded moments and they are given by
\begin{equation}
\label{equation:properties_T1_moments}
\E[T_1^k]=\frac{(-2)^kk!}{(2k)!}\sum_{j=1}^p (-1)^j \sum_{\underset{n_i\in\N^*}{n_1+\cdots+n_j=p}} \binom{2p}{2n_1,\ldots,2n_j}.
\end{equation}
In particular, $\E[T_1]=1$, $\E[T_1^2]=\frac{5}{3}$ and $\Var(T_1)=\frac{2}{3}$.
Finally, for a fixed $\varepsilon_0\in \left]0,\frac{\pi^2}{16}\right[$, for all $\varepsilon \in]0,\varepsilon_0]$ and $p\geq 0$, we have the estimate
\begin{equation}
\label{equation:properties_T1_exponential_estimate}
\E[e^{\varepsilon T_N}(\varepsilon T_N)^p] \leq Ce^{C \varepsilon N}(\varepsilon N)^p.
\end{equation}
\end{proposition}

The law of the first revolution time $T_1$ has an analytic density but there is no closed formula for it. It can be numerically approximated accurately by inverting the Laplace transform.
In Figure \ref{figure:Law_TN}, we observe the convergence in law of $T_N$ to a Gaussian variable according to the central limit theorem.
\begin{figure}
	\begin{center}
		\includegraphics[scale=0.6]{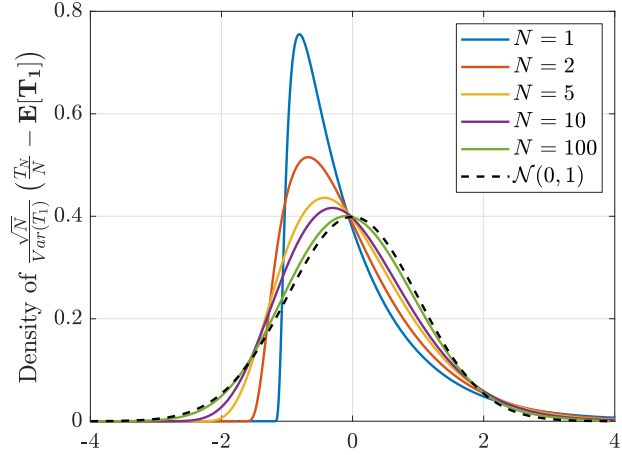}
		\caption{Convergence in law of $\frac{\sqrt{N}}{Var(T_1)}\left(\frac{T_N}{N}-\E[T_1]\right)$ to a standard Gaussian random variable.}
		\label{figure:Law_TN}
	\end{center}
\end{figure}

\begin{proof}[Proof of Proposition \ref{proposition:properties_T1}]
The first properties can be deduced from \cite[Chap.\ts 2.3]{Revuz94cma}, where the Laplace transform formula is obtained with an analytic continuation of the equality $\E[e^{-xT_1}]=\frac{1}{\cosh(\sqrt{2x})}$ for $x>0$.
Comparing the Taylor expansion of $\E[e^{xT_1}]$ and $\frac{1}{\cos(\sqrt{2x})}$ yields \eqref{equation:properties_T1_moments}.
The estimate \eqref{equation:properties_T1_exponential_estimate} is proved as follows
\begin{align*}
\E[e^{\varepsilon T_N}(\varepsilon T_N)^p] &\leq \E[e^{2\varepsilon T_N}]^{1/2} \E[(\varepsilon T_N)^{2p}]^{1/2}
= \E[e^{2\varepsilon T_1}]^{N/2} \varepsilon^p \E[T_N^{2p}]^{1/2}\\
&\leq \E[e^{2\varepsilon_0 T_1}]^{\varepsilon N/2\varepsilon_0} (\varepsilon N)^p \E[T_1^{2p}]^{1/2}
\leq C e^{C \varepsilon N} (\varepsilon N)^p,
\end{align*}
where we used first the Cauchy-Schwarz inequality and then twice the Jensen inequality.
\end{proof}

For developing an algorithm for the weak error, it is useful to know the moments of the random variables appearing in the discretization, that are costly to simulate numerically, in order to replace them with cheap discrete approximations with the same first and second moments. This is the goal of the following proposition.
\begin{proposition}
\label{proposition:moments_alpha_beta}
The following random variables
$$
\begin{array}{rl}
\alpha_k^N&=\frac{1}{N}\int_0^{T_N} e^{2i\pi k W(s)}ds\\
\beta_{p,k}^N&=\frac{1}{N^2} \int_0^{T_N} e^{2i\pi pW(s)}\int_0^{s} e^{2i\pi kW(r)} dr ds\\
\widetilde{\beta}_{p,k}^N&=\frac{1}{N^2} \int_0^{T_N} e^{2i\pi pW(s)} \left(
\int_0^{s} e^{2i\pi kW(r)} -\frac{1}{2} \int_0^{T_N} e^{2i\pi kW(r)} dr \right) ds
= \beta_{p,k}^N-\frac{\alpha_p^N\alpha_k^N}{2}
\end{array}
$$
satisfy
\vspace{-18pt}
\begin{multicols}{2}
$$
\begin{array}{rl}
\E[\alpha_k^N]&=\delta_k
=\left\{
\begin{array}{l}
1 \text{ if } k=0\\
0 \text{ else}
\end{array}
\right.\\
\E[\beta_{p,k}^N]&=\left\{
\begin{array}{l}
\frac{1}{2}+\frac{1}{3N} \text{ if } p=k=0\\
\frac{1}{2\pi^2k^2N} \text{ if } p=0,\:  k\neq 0\\
\frac{-1}{2\pi^2p^2N} \text{ if } p\neq 0,\:  k= 0\\
\frac{1}{2\pi^2p^2N} \text{ if } p+k= 0,\: p,k\neq 0\\
0 \text{ else}
\end{array}
\right.
\end{array}
\hskip-0.5cm
\begin{array}{rl}
\E[\alpha_p^N\alpha_k^N]&=\left\{
\begin{array}{l}
1+\frac{2}{3N} \text{ if } p=k=0\\
\frac{1}{\pi^2p^2N} \text{ if } p+k= 0,\: p,k\neq 0\\
0 \text{ else}
\end{array}
\right.\\
\E[\widetilde{\beta}_{p,k}^N]&=\left\{
\begin{array}{l}
\frac{1}{2\pi^2k^2N} \text{ if } p=0,\:  k\neq 0\\
\frac{-1}{2\pi^2p^2N} \text{ if } p\neq 0,\:  k= 0\\
0 \text{ else}
\end{array}
\right.\\
\end{array}
$$
\end{multicols}
\end{proposition}

\begin{proof}
Let $k\neq 0$ (the case $k=0$ is straightforward using Proposition \ref{proposition:properties_T1}), then the Itô formula applied to $e^{2i\pi k W(s)}$ gives
$$\frac{1}{2\pi^2 k^2 N} (e^{2i\pi k W(t)}-1) =  \frac{i}{\pi k N} \int_0^{t} e^{2i\pi k W(s)} dW(s) -\frac{1}{N} \int_0^{t} e^{2i\pi k W(s)} ds,$$
which yields at time $t=T_N$,
$$\alpha_k^N= \frac{i}{\pi kN} \int_0^{T_N} e^{2i\pi k W(s)} dW(s).$$
Then $t\mapsto \int_0^{t} e^{2i\pi k W(s)} dW(s)$ is a martingale, so by the Doob theorem, as $t\wedge T_N$ is finite, $\E[\int_0^{t\wedge T_N} e^{2i\pi k W(s)} dW(s)]=0$ for all $t$. The dominated convergence theorem for stochastic integrals allows to take the limit $t\rightarrow\infty$ and yields $\E[\alpha_k^N]=0$.

For the coefficients $\beta_{p,k}^N$, let $(p,k)\neq (0,0)$ (the case $p=k=0$ is obtained straightforwardly using Proposition \ref{proposition:properties_T1}), we use the Itô formula on $e^{2i\pi p W(s)}$ and we integrate from $s$ to $T_N$,
$$1-e^{2i\pi p W(s)} = 2i\pi p \int_s^{T_N} e^{2i\pi p W(r)} dW(r) -2\pi^2 p^2 \int_s^{T_N} e^{2i\pi p W(r)} dr.$$
Then, multiplying by $\frac{1}{N^2}e^{2i\pi k W(s)}$ and integrating from $0$ to $T_N$ yields
\begin{align*}
\frac{\alpha_k^N-\alpha_{p+k}^N}{N} &= \frac{2i\pi p}{N^2} \int_0^{T_N} e^{2i\pi k W(s)} \int_s^{T_N} e^{2i\pi p W(r)} dW(r) ds\\
&-\frac{2\pi^2 p^2}{N^2} \int_0^{T_N} e^{2i\pi k W(s)} \int_s^{T_N} e^{2i\pi p W(r)} dr ds.
\end{align*}
Using the stochastic Fubini theorem, 
we deduce
$$\int_0^{T_N} e^{2i\pi k W(s)} \int_s^{T_N} e^{2i\pi p W(r)} dW(r) ds
=\int_0^{T_N} e^{2i\pi p W(r)} \int_0^r e^{2i\pi k W(s)} ds dW(r),$$
which has zero average by the same arguments as before.
Also by the Fubini theorem for stochastic integrals, $\beta_{p,k}^N=\frac{1}{N^2}\int_0^{T_N} e^{2i\pi k W(s)} \int_s^{T_N} e^{2i\pi p W(r)} dr ds$, so that we get if $p\neq 0$,
$$\E[\beta_{p,k}^N]=\frac{\delta_{p+k}-\delta_k}{2\pi^2 p^2 N}.$$
The case $p = 0$ is obtained by integrating by parts and using the same arguments.
Indeed, we find
$$\beta_{0,k}^N=\frac{T_N}{N}\alpha_k^N-\beta_{k,0}^N
=\frac{i}{\pi kN^2} \int_0^{T_N} T_N e^{2i\pi k W(s)} dW(s)-\beta_{k,0}^N,$$
and $\E[\beta_{0,k}^N]=-\E[\beta_{k,0}^N]$.
Finally, the moments $\E[\alpha_p^N\alpha_k^N]$ are computed via the equality $\beta_{p,k}^N+\beta_{k,p}^N=\alpha_p^N\alpha_k^N$. Then we obtain $\E[\widetilde{\beta}_{p,k}^N]$ from the formula $\widetilde{\beta}_{p,k}^N=\beta_{p,k}^N-\frac{\alpha_p^N\alpha_k^N}{2}$.
\end{proof}

\begin{remark*}[Stochastic Fourier series]
Let $f$ be a $L^2$ function on $]0,1[$ extended on $\R$ by 1-periodicity, whose Fourier coefficients are denoted as $(c_k)_{k\in\Z}$. Then we deduce from Proposition \ref{proposition:moments_alpha_beta} the following equalities, where the second one is the stochastic version of the Bessel-Parseval theorem,
\begin{align*}
\E\left[\int_0^{T_1} f(W(s)) ds\right]&=c_0=\int_0^1 f(\theta) d\theta,\\
\E\left[\int_0^{T_1} \abs{f(W(s))}^2 ds\right]&=\sum_{k\in \Z} \abs{c_k}^2,\\
\E\left[\abs{\int_0^{T_1} f(W(s)) ds}^2\right]&=\frac{5\abs{c_0}^2}{3}+\sum_{k\in\Z^*} \frac{\abs{c_k}^2}{\pi^2 k^2}.
\end{align*}
\end{remark*}

\subsection{Asymptotic expansion at revolution times and limit model}
\label{section:limit_model}

With the results of Subsection \ref{section:probabilistic_properties}, it is now possible to evaluate the local expansions \eqref{equation:strong_Taylor_development} at revolution times.
To approximate numerically the integrals appearing in equation \eqref{equation:strong_Taylor_development} without evaluating $F$ and $F'$ too many times, we first replace the 1-periodic functions $g_\theta^0(y)$ and $g_\theta^1(y)(z)$ defined in \eqref{equation:def_g_theta} by their associated Fourier series with Fourier coefficients $(c_k^0(y))_{k\in\Z}$ and $(c_p^1(y))_{p\in\Z}$.
We define the following approximation of $\varphi_{\varepsilon,t}(y)$,
\begin{align}
\label{equation:Taylor_expansion_Fourier}
\psi_{\varepsilon,t}(y)&=e^{AW(t)}y
+\varepsilon \sum_{k\in \Z} e^{AW(t)} c_k^0(y)  \int_0^t e^{2i\pi k W(s)} ds\\
&+\varepsilon^2 \sum_{p,k\in \Z} e^{AW(t)} c_p^1(y) \left( c_k^0(y) \int_0^t \int_0^s e^{2i\pi p W(s)} e^{2i\pi k W(r)} dr ds \right). \nonumber
\end{align}
Notice that $c_k^0(y)\in\C^d $ and $c_p^1(y)=(c_k^0)'(y)\in\C^{d\times d}$ but $\psi_{\varepsilon,t}(y)\in \R^d$.
We now evaluate this function $\psi_{\varepsilon,t}(y)$ at time $t=T_N$ to get a second order strong approximation.

\begin{proposition}
\label{proposition:strong_Taylor_development_hitting_times}
We define the following quantity
$$
\psi_{\varepsilon,N}(y)=y+H\sum_{k\in\Z} c_k^0(y) \alpha_k^N+H^2\sum_{p,k\in \Z} c_p^1(y)(c_k^0(y)) \beta_{p,k}^N,
$$
where
$(c_k^0(y))_{k\in\Z}$ and $(c_p^1(y))_{p\in\Z}$ are the Fourier coefficients of the 1-periodic functions $g_\theta^0(y)$ and $g_\theta^1(y)$ defined in \eqref{equation:def_g_theta},
$\alpha_k^N$, $\beta_{p,k}^N$ are the random variables defined in Proposition \ref{proposition:moments_alpha_beta} and $y\in\R^d$ is deterministic.
Under Assumption \ref{assumption:F_Lipschitz}, for all test function $\phi \in \CC^3_{P}$, there exists $H_0>0$ such that for all $H=N\varepsilon\leq H_0$, the following estimates hold, where $C$ and $K$ are independent of $\varepsilon$ and $N$,
\begin{align}
\label{equation:strong_local_order_2_estimate}
\E\left[\abs{\varphi_{\varepsilon,T_N}(y)-\psi_{\varepsilon,N}(y)}^2\right]^{1/2}&\leq C(1+\abs{y}^K)H^3,\\
\label{equation:weak_local_order_2_estimate}
\abs{\E[\phi(\varphi_{\varepsilon,T_N}(y))|y]-\E[\phi(\psi_{\varepsilon,N}(y))]}&\leq C(1+\abs{y}^K)H^3,
\end{align}
that is, $\psi_{\varepsilon,N}(y)$ is a numerical approximation of $\varphi_{\varepsilon,T_N}(y)$ of strong/weak local order two.
\end{proposition}

\begin{proof}
Inequality \eqref{equation:strong_local_order_2_estimate} is a straightforward consequence of Proposition \ref{proposition:properties_T1} when evaluating the estimates of Proposition \ref{proposition:strong_Taylor_development} at time $T_N$.
For the weak local estimate \eqref{equation:weak_local_order_2_estimate}, using inequality \eqref{equation:strong_local_order_2_estimate}, the mean value inequality, Lemma \ref{lemma:regularity_varphi} and equations \eqref{equation:Taylor_bound_1} and \eqref{equation:Taylor_bound_2}, we get
\begin{align*}
\E\Big[\Big|\phi(\varphi_{\varepsilon,T_N}(y))&-\phi(\psi_{\varepsilon,T_N}^j(y))\Big|\Big]\\
&\leq \E\left[\sup_{x \in [\varphi_{\varepsilon,T_N}(y) , \psi_{\varepsilon,T_N}^j(y)]} \abs{\phi'(x)} \abs{\varphi_{\varepsilon,T_N}(y)-\psi_{\varepsilon,T_N}^j(y)}\right]\\
&\leq C(1+\abs{y}^K)\E\left[e^{C\varepsilon T_N}(\varepsilon T_N)^{j+1} \sup_{x \in [\varphi_{\varepsilon,T_N}(y) , \psi_{\varepsilon,T_N}^j(y)]} (1+\abs{x}^p)\right]\\
&\leq C(1+\abs{y}^K)\E\left[(\varepsilon T_N)^{j+1} e^{C\varepsilon T_N} \left(1+\abs{\varphi_{\varepsilon,T_N}(y)}^p+\abs{\psi_{\varepsilon,T_N}^j(y)}^p\right)\right]\\
&\leq C(1+\abs{y}^K)\E\left[(\varepsilon T_N)^{j+1} e^{C\varepsilon T_N}\right].
\end{align*}
Finally we obtain inequality \eqref{equation:weak_local_order_2_estimate} by taking $H$ small enough so that we can apply Proposition \ref{proposition:properties_T1}.
\end{proof}

For a fixed $T=N\varepsilon$, when $\varepsilon\rightarrow 0$ (or equivalently $N\rightarrow\infty$), the solution of \eqref{equation:NLS_rescaled} evaluated at stroboscopic times $T_N=T_{T\varepsilon^{-1}}$ converges weakly to the solution of a deterministic ODE, that involves only the first mode $c^0_0=\langle g^0 \rangle=\int_0^1 g^0_\theta d\theta$ of $g^0$. This asymptotic model is the same one as for the deterministic equation \eqref{equation:highly_oscillatory_ODE}. The proof is postponed to Subsection \ref{section:global_error_proof}.
\begin{proposition}[Asymptotic model]
\label{proposition:asymptotic_model}
Under Assumption \ref{assumption:F_Lipschitz}, for $T>0$, the solution $\varphi_{\varepsilon,T_{T\varepsilon^{-1}}}(X_0)$ (for $\varepsilon$ such that $T\varepsilon^{-1}$ is an integer) of equation \eqref{equation:NLS} converges weakly when $\varepsilon\rightarrow 0$ to the solution at time $T$ of
\begin{equation}
\label{equation:NLS_asymptotic}
\frac{d y_{t}}{dt}=\langle g^0 \rangle (y_{t}), \ y_{0}=X_0,
\end{equation}
that is, for all test function $\phi \in \CC^3_P$,
$$\lim_{\varepsilon\rightarrow 0} \abs{\E[\phi(\varphi_{\varepsilon,T_{T\varepsilon^{-1}}}(X_0))]-\E[\phi(y_{T})]}=0.$$
\end{proposition}

\begin{remark}
It can be proven using the results of Section \ref{section:convergence_theorem} that the solution of the asymptotic model (Proposition \ref{proposition:asymptotic_model}) is an order one weak approximation of $X(\varepsilon T_{Nm})$ for $m\geq 0$ and $X$ solution of equation \eqref{equation:NLS}.
We deduce the following simple one-step explicit deterministic integrator that corresponds to the Euler method applied to equation \eqref{equation:NLS_asymptotic},
\begin{equation}
\label{algorithm:weak_order_1}
y_0 = X_0,\ y_{m+1} = y_m + H c_0^0(y_m).
\end{equation}
Its cost is independent of $\varepsilon$ and $N$, and it has weak order one w.r.t.\ts $H$, that is for all $m\geq 0$, $\E[\phi(\varphi_{\varepsilon,T_{Nm}}(X_0))]-\E[\phi(y_m)]= \OO(H)$.
\end{remark}

\subsection{Construction of the second order integrators}
\label{section:stiff_integrators}


To obtain an integrator of weak order two with a cost independent of $\varepsilon$ and $N$, we truncate the local expansion of Proposition \ref{proposition:strong_Taylor_development_hitting_times}. We also replace the involved random variables with cheap discrete random variables with the same first and second moments.
To simulate the random variable $\alpha_k^N$ with discrete random variables $\widehat{\alpha}_k^{N}$ with the same first and second moments, we introduce a set $(\xi_k)_{k\in\N}$ of independent random variables, such that $\P(\xi_k = \pm 1)=\frac{1}{2}$, the covariance matrix $(C_\alpha^N)_{p,k}$ such that
$$(C_\alpha^N)_{2p-1:2p,2k-1:2k}=\begin{pmatrix}
\Cov(\Real(\alpha_p^N),\Real(\alpha_k^N)) & \Cov(\Real(\alpha_p^N),\Imag(\alpha_k^N))\\
\Cov(\Imag(\alpha_p^N),\Real(\alpha_k^N)) & \Cov(\Imag(\alpha_p^N),\Imag(\alpha_k^N))
\end{pmatrix},$$
and $\Gamma^N$ its square root. Then, $\widehat{\alpha}_k^{N}$ is defined for $k\geq 0$ as
$$\widehat{\alpha}_k^N=\delta_k+\sum_{l\in\N} (\Gamma^N_{2k-1,l}+i\Gamma^N_{2k,l}) \xi_l \quad \text{with} \quad \delta_k
=\left\{
\begin{array}{l}
1 \text{ if } k=0\\
0 \text{ else}
\end{array}
\right.$$
and we fix $\widehat{\alpha}_k^N=\overline{\widehat{\alpha}_{-k}^N}$ for $k<0$ (so that the solution stays real while still having the good moments).
We also define $\widehat{\beta}_{p,k}^N=\E[\beta_{p,k}^N]$ with the values of Proposition \ref{proposition:moments_alpha_beta}.
Doing so yields Method \ref{algorithm:weak_order_2}.

For Method \ref{algorithm:geometric_weak_order_2}, we adapt Method \ref{algorithm:weak_order_2} in the spirit of the middle point scheme for ODEs (see \cite[Chap.\ts IV]{Hairer06gni}) so that it preserves any quadratic invariant.
We also replace $\widehat{\beta}_{p,k}^N$ by $\widehat{\widetilde{\beta}}_{p,k}^N=\E[\widetilde{\beta}_{p,k}^N]$, using the values of Proposition \ref{proposition:moments_alpha_beta}.

\begin{remark}
The methodology presented in Section \ref{section:construction_algorithm} can be generalised to any order.
Thus, under more regularity assumptions on $F$, it is possible to build algorithms similar to Method \ref{algorithm:weak_order_2} of any weak order and that are still robust with respect to the stiffness $\varepsilon$.
For order 3, Method \ref{algorithm:weak_order_2} becomes
\begin{align*}
Y_{m+1} &= Y_m
+H \sum_{k\in \Z} c_k^0(Y_m) \widehat{\alpha}_k^N
+H^2 \sum_{p,k\in \Z} c_p^1(Y_m) (c_k^0(Y_m)) \widehat{\beta}_{p,k}^N\\
&+H^3 \sum_{l,p,k\in \Z} c_l^1 (c_p^1 (c_k^0))(Y_m) \widehat{\gamma}_{l,p,k}^{(1),N}
+c_l^2 (c_p^0,c_k^0)(Y_m) \widehat{\gamma}_{l,p,k}^{(2),N}
\end{align*}
with the new random variables
\begin{align*}
\gamma_{l,p,k}^{(1),N}&=\frac{1}{N^2} \int_0^{T_N} e^{2i\pi qW(s)}\int_0^{s} e^{2i\pi pW(r)} \int_0^{r} e^{2i\pi kW(q)} dq dr ds\\
\gamma_{l,p,k}^{(2),N}&=\frac{1}{2 N^2} \int_0^{T_N} e^{2i\pi qW(s)}\int_0^{s} e^{2i\pi pW(r)} \int_0^{s} e^{2i\pi kW(q)} dq dr ds,
\end{align*}
and where the discrete random variables share the same moments up to order 3 for the $\widehat{\alpha}_k^N$, order 2 for the $\widehat{\beta}_{p,k}^N$, and order 1 for the $\widehat{\gamma}_{l,p,k}^{(i),N}$.
It is also possible to generalise Method \ref{algorithm:geometric_weak_order_2} up to any order in the spirit of the middle-point scheme, but the construction of discrete random variables allowing the preservation of quadratic invariants is not obvious for higher orders (although backward error analysis guarantees the preservation of quadratic invariants for the exact random variables based on $W$).

\end{remark}

\section{Weak convergence analysis}
\label{section:convergence_theorem}

This section focuses on the proofs of the following two theorems, showing the order two convergence of Methods \ref{algorithm:weak_order_2} and \ref{algorithm:geometric_weak_order_2}.

\begin{theorem}
\label{theorem:convergence_algorithm_order_2}
Assume that the Fourier coefficients $c_k^0$, $c_p^1$ of
$g_\theta^0$ and $g_\theta^1$ in \eqref{equation:def_g_theta} are non-zero only for $-K_t /2\leq k,p< K_t /2$.
Then, under Assumption \ref{assumption:F_Lipschitz}, Method \ref{algorithm:weak_order_2} has weak order two, that is, for all $T>0$, for all test function $\phi \in \CC^3_P$, there exists $H_0>0$ such that for all $H\leq H_0$, for all $m\geq 0$ such that $mN\varepsilon=mH\leq T$, there exists two positive constants $K$ and $C$ both independent of $\varepsilon$, $N$ and $K_t$ such that
\begin{equation}
\label{equation:theorem_cv_estimate}
\abs{\E[\phi(\varphi_{\varepsilon,T_{Nm}}(X_0))]-\E[\phi(Y_m)]}\leq CH^2(1+\E[\abs{X_0}^{K}]).
\end{equation}
\end{theorem}

\begin{theorem}
\label{theorem:convergence_algorithm_order_2_geometric}
Assume that the Fourier coefficients $c_k^0$, $c_p^1$ of
$g_\theta^0$ and $g_\theta^1$ in \eqref{equation:def_g_theta} are non-zero only for $-K_t /2\leq k,p < K_t /2$.
Then, under Assumption \ref{assumption:F_Lipschitz}, if $c_0^1(c_k^0)$ and $c_p^1(c_0^0)$ are Lipschitz continuous uniformly in $k$ and $p$, Method \ref{algorithm:geometric_weak_order_2} is well defined and has weak order two (i.e.\ts it satisfies an estimate of the form \eqref{equation:theorem_cv_estimate}).
In addition, if for a fixed symmetric matrix $ \invariantmatrix \in \R^{d\times d}$, the quantity $\invariant(y)=\frac{1}{2}y^T \invariantmatrix y$ is preserved by equation \eqref{equation:NLS}, then Method \ref{algorithm:geometric_weak_order_2} also preserves the invariant $\invariant(y)=\frac{1}{2}y^T \invariantmatrix y$, that is, the solution $Y_{m+1}$ of equation \eqref{equation:geometric_order_2_scheme} satisfies $\invariant(Y_{m+1})=\invariant(Y_m)$.
\end{theorem}

These two theorems focus on approximating the exact solution of equation \eqref{equation:NLS} at the revolution times $\varepsilon T_{Nm}$, $m=0,1,\dots$, but one could compute an approximation at different times by composing with other methods at the end of the integration.



\begin{remark} \label{rem:trunc}
Since the error constant $C$ in \eqref{equation:theorem_cv_estimate} is independent of the number $K_t$ of Fourier modes, we emphasize that Theorem \ref{theorem:convergence_algorithm_order_2} and Theorem \ref{theorem:convergence_algorithm_order_2_geometric} remain
valid for infinitely many modes ($K_t\rightarrow \infty$).
In addition, assuming that $F$ is of class $C^{s+1}_P$ yields a truncation error of the Fourier series in \eqref{equation:def_g_theta} of size $\OO((1+\abs{y}^K)K_t ^{-s})$ (see e.g.\ts \cite[Sect.\ts III.1.3]{Lubich08fqt}), and if $g_\theta^0$ is assumed analytic in $\theta$ (for example if $F$ is a polynomial), this error becomes exponentially small as $\OO((1+\abs{y})e^{-c K_t })$.
For simplicity of the analysis, we thus assumed in Theorem \ref{theorem:convergence_algorithm_order_2} and Theorem \ref{theorem:convergence_algorithm_order_2_geometric} that $g_\theta^0$ and $g_\theta^1$ have a finite number $K_t $ on non-zero Fourier modes in \eqref{equation:def_g_theta}.
If this assumption does not hold, the truncation errors $\OO((1+\abs{y}^K)K_t ^{-s})$ or $\OO((1+\abs{y})e^{-c K_t })$ should be added in the right-hand side of the error estimate \eqref{equation:theorem_cv_estimate}.
Let us prove it in the analytic case. We first apply the change of variable $\widetilde{\varphi}_{\varepsilon,t}(y)=e^{-AW(t)}\varphi_{\varepsilon,t}(y)$ that has no effect at time $t=T_{Nm}$. We now have to compare the two solutions of the following integral formulations
\begin{align*}
\widetilde{\varphi}_{\varepsilon,t}(y)&=y+\varepsilon\int_0^t g_{W(s)}^0(\widetilde{\varphi}_{\varepsilon,s}(y))ds,\\
\widetilde{\varphi}_{\varepsilon,t}^{(K_t)}(y)&=y+\varepsilon\int_0^t \sum_{k=-K_t/2}^{K_t/2-1} c_k^0(\widetilde{\varphi}_{\varepsilon,s}^{(K_t)}(y))e^{2i\pi k W(s)}ds.
\end{align*}
Using the truncation estimates that we previously discussed and the Lipschitz property of $g_\theta^0$, one gets
\begin{align*}
\abs{\widetilde{\varphi}_{\varepsilon,t}-\widetilde{\varphi}_{\varepsilon,t}^{(K_t)}}(y)
&\leq \varepsilon \int_0^t \abs{g_{W(s)}^0(\widetilde{\varphi}_{\varepsilon,s}(y))-g_{W(s)}^0(\widetilde{\varphi}_{\varepsilon,s}^{(K_t)}(y))} ds
+C\varepsilon t e^{-c K_t} \sup_{[0,1]} \abs{g^0(y)}\\
&\leq C\varepsilon \int_0^t \abs{\widetilde{\varphi}_{\varepsilon,s}(y)-\widetilde{\varphi}_{\varepsilon,s}^{(K_t)}(y)} ds
+C\varepsilon t (1+\abs{y})e^{-c K_t}.
\end{align*}
The Gronwall lemma and Proposition \ref{proposition:properties_T1} yield for $mN\varepsilon  \leq T$,
$$\E\left[\abs{\widetilde{\varphi}_{\varepsilon,T_{Nm}}(y)-\widetilde{\varphi}_{\varepsilon,T_{Nm}}^{(K_t)}(y)}^2\right]^{1/2}
\leq C(1+\abs{y})e^{-c K_t}.$$
\end{remark}

The structure of the convergence proof is similar to the one for standard SDE integrators, see e.g.\ts \cite[Chap.\ts 2]{Milstein04snf}, but one has to be cautious because our solution is evaluated at stochastic times and the error constants should not depend on $\varepsilon$ or $N$.

\subsection{Boundedness of the numerical moments}

\begin{proposition}[Bounded moments for the integrator \eqref{equation:order_2_scheme}]
\label{proposition:bounded_moments}
Assume that for $y\in\R^d$, the numerical integrator $\widehat{\psi}_{\varepsilon,N}(y)$ is given by
\begin{equation}
\label{equation:proof_integrator_order_2}
\widehat{\psi}_{\varepsilon,N}(y) = y
+H \sum_{k\in\Z} c_k^0(y) \widehat{\alpha}_k^{N}
+H^2 \sum_{p,k\in\Z} c_p^1(y) (c_k^0(y)) \widehat{\beta}_{p,k}^N,
\end{equation}
where $\widehat{\alpha}_k^N$, $\widehat{\beta}_{p,k}^N$ are random variables defined such that for all $q>0$, $\E\left[\left(\sum_{k} \frac{\abs{\widehat{\alpha}_k^N}^2}{k^2}\right)^{q}\right]$ and $\E\left[\left(\sum_{p,k} \frac{\abs{\widehat{\beta}_{p,k}^N}^2}{k^2}\right)^{q}\right]$ are bounded uniformly in $N$.
Then, under Assumption \ref{assumption:F_Lipschitz} and if $\abs{y}$ has bounded moments, for any $T>0$, for all $m$, $H$ such that $m\varepsilon N=mH\leq T$, for all $q>0$, we have $\E[\abs{\widehat{\psi}_{\varepsilon,N}^m(y)}^{2q}]\leq C_q(1+\E[\abs{y}^{2q}])$, where $C_q$ is independent of $m$, $\varepsilon$ and $N$.
\end{proposition}

\begin{proof}
We first prove
\begin{equation}
\label{equation:inequality_bounded_moments}
\abs{\widehat{\psi}_{\varepsilon,N}(y)-y}\leq CH(1+\abs{y})M_N,
\end{equation}
where $\E[(M_N)^{2q}]\leq C_q$ for all $q>0$. We have
\begin{align*}
\abs{\widehat{\psi}_{\varepsilon,N}(y)-y}&= H\abs{\sum_{k\in\Z} c_k^0(y) \widehat{\alpha}_k^N+H\sum_{p,k\in\Z} c_p^1(y)(c_k^0(y)) \widehat{\beta}_{p,k}^N}\\
&\leq C H \left(M_N^{(0)} \sqrt{\sum_{k\in\Z} k^2 \abs{c_k^0(y)}^2}+ M_N^{(1)} \sqrt{\sum_{p\in\Z} \abs{c_p^1(y)}^2} \sqrt{\sum_{k\in\Z} k^2 \abs{c_k^0(y)}^2} \right)
\end{align*}
where $M_N^{(0)}=\sqrt{\sum_k \frac{\abs{\widehat{\alpha}_k^N}^2}{k^2}}$ and $M_N^{(l)}=\sqrt{\sum_{p,k} \frac{\abs{\widehat{\beta}_{p,k}^N}^2}{k^2}}$ have moments bounded uniformly in $N$.
Then using the Bessel-Parseval theorem, we get $\sum_k k^2 \abs{c_k^0(y)}^2=\int_0^1 \abs{\partial_\theta g_\theta(y)^0}^2 d\theta$.
Assumption \ref{assumption:F_Lipschitz} yields $\abs{\partial_\theta g_\theta^0(y)}\leq C(1+\abs{y})$. Then, the Bessel-Parseval theorem applied on $g_\theta^1$ gives $\sqrt{\sum_p \abs{c_p^1(y)}^2}\leq C$, hence the result.

We define $\Delta \psi_m=\widehat{\psi}_{\varepsilon,N}^{m+1}(y)-\widehat{\psi}_{\varepsilon,N}^{m}(y)=(\widehat{\psi}_{\varepsilon,N}-\Id)(\widehat{\psi}_{\varepsilon,N}^{m}(y))$, then
$$(\widehat{\psi}_{\varepsilon,N}^{m+1}(y))^{2q}=(\widehat{\psi}_{\varepsilon,N}^{m}(y))^{2q}
+\sum_{j=1}^{2q} \binom{2l}{j}(\widehat{\psi}_{\varepsilon,N}^{m}(y))^{2q-j} \Delta \psi_m^j . $$
Equation \eqref{equation:inequality_bounded_moments} and the bounded moments of $M_N$ give
\begin{align*}
\abs{\E\left[(\widehat{\psi}_{\varepsilon,N}^{m}(y))^{2q-j} \Delta \psi_m^j\right]} &\leq \E\left[\abs{\widehat{\psi}_{\varepsilon,N}^{m}(y)}^{2q-j} CH^j(1+\abs{\widehat{\psi}_{\varepsilon,N}^{m}(y)}^j)M_N^j\right]\\
&\leq C_q H\left(1+\E\left[\abs{\widehat{\psi}_{\varepsilon,N}^{m}(y)}^{2q}\right]\right).
\end{align*}
We deduce $$1+\E\left[\abs{\widehat{\psi}_{\varepsilon,N}^{m+1}(y)}^{2q}\right]\leq e^{C_q H} \left(1+\E\left[\abs{\widehat{\psi}_{\varepsilon,N}^{m}(y)}^{2q}\right]\right)$$
and by induction $\E\left[\abs{\widehat{\psi}_{\varepsilon,N}^{m}(y)}^{2q}\right]\leq e^{C_q mH} (1+\E[\abs{y}^{2q}])\leq e^{C_q T} (1+\E[\abs{y}^{2q}])$.
\end{proof}

\begin{proposition}[Bounded moments for the integrator \eqref{equation:geometric_order_2_scheme}]
\label{proposition:bounded_moments_geometric}
Assume that for $y\in\R^d$, the numerical scheme $\widehat{\psi}_{\varepsilon,N}(y)$ satisfies 
\begin{align}
\label{equation:proof_geometric_integrator_order_2}
\widehat{\psi}_{\varepsilon,N}(y) &= y
+H \sum_{k\in\Z} c_k^0\left(\frac{y+\widehat{\psi}_{\varepsilon,N}(y)}{2}\right) \widehat{\alpha}_k^{N}\\
&+H^2 \sum_{p,k\in\Z} c_p^1\left(\frac{y+\widehat{\psi}_{\varepsilon,N}(y)}{2}\right) \left(c_k^0\left(\frac{y+\widehat{\psi}_{\varepsilon,N}(y)}{2}\right)\right) \widehat{\widetilde{\beta}}_{p,k}^N, \nonumber
\end{align}
where $\widehat{\alpha}_k^N$, $\widehat{\widetilde{\beta}}_{p,k}^N$ are random variables defined such that for all $q>0$, $\sum_{k} \abs{\widehat{\alpha}_k^N}$, $\sum_{p,k} \abs{\widehat{\widetilde{\beta}}_{p,k}^N}$, $\E\left[\left(\sum_k \frac{\abs{\widehat{\alpha}_k^N}^2}{k^2}\right)^{q}\right]$ and $\E\bigg[\bigg(\sum_{p,k} \frac{\abs{\widehat{\widetilde{\beta}}_{p,k}^N}^2}{k^2}\bigg)^{q}\bigg]$ are bounded uniformly in $N$.
Then, under Assumption \ref{assumption:F_Lipschitz} and if $\abs{y}$ has bounded moments, for $H_0$ small enough and any $T>0$, for all $m$, $H$ such that $m\varepsilon N=mH\leq T$ and $H\leq H_0$, for all $q>0$, we have $\E[\abs{\widehat{\psi}_{\varepsilon,N}^m(y)}^{2q}]\leq C_q(1+\E[\abs{y}^{2q}])$, where $C_q$ is independent of $m$, $\varepsilon$ and $N$.
\end{proposition}

\begin{proof}
We prove an equivalent of the estimate \eqref{equation:inequality_bounded_moments} for $\widehat{\psi}_{\varepsilon,N}(y)$.
The growth properties of the Fourier coefficients yield 
\begin{align*}
\abs{\widehat{\psi}_{\varepsilon,N}(y)-y}
&\leq 
CH \left(\sum_k \abs{c_k^0}(y) \abs{\widehat{\alpha}_k^N}+\sum_{p,k} \abs{c_k^0}(y)  \abs{\widehat{\widetilde{\beta}}_{p,k}^N}\right)\\
&+CH \left(\sum_k \abs{\widehat{\alpha}_k^N}+\sum_{p,k} \abs{\widehat{\widetilde{\beta}}_{p,k}^N}\right)\abs{\widehat{\psi}_{\varepsilon,N}(y)-y},
\end{align*}
As $\sum_k \abs{\widehat{\alpha}_k^N}+\sum_{p,k} \abs{\widehat{\widetilde{\beta}}_{p,k}^N}$ is bounded, using the same estimates as in the proof of Proposition \ref{proposition:bounded_moments}, we get for all $H\leq H_0$ small enough,
\begin{equation}
\label{equation:inequality_bounded_moments_geometric}
\abs{\widehat{\psi}_{\varepsilon,N}(y)-y} \leq CH(1+\abs{y})M_N,
\end{equation}
where $M_N$ has bounded moments.
The remaining of the proof is the same as in the proof of Proposition \ref{proposition:bounded_moments}.
\end{proof}

\subsection{Local weak error}

\begin{proposition}[Local error estimate]
\label{proposition:local_weak_error}
Assume that for $y\in\R^d$ deterministic, the numerical scheme can be written as
$$
\widehat{\psi}_{\varepsilon,N}(y)=y+H\sum_{k\in\Z} c_k^0(y) \widehat{\alpha}_k^N+H^2\sum_{p,k\in\Z} c_p^1(y)(c_k^0(y)) \widehat{\beta}_{p,k}^N+R,
$$
where $\E[\abs{R}]\leq C(1+\abs{y}^K)H^3$ and $\widehat{\alpha}_k^N\in\C$, $\widehat{\beta}_{p,k}^N\in\R$ are random variables such that $\widehat{\alpha}_k^N=\overline{\widehat{\alpha}_{-k}^N}$ and
$$
\E[\widehat{\alpha}_k^N]=\E[\alpha_k^N],\
\E[\widehat{\beta}_{p,k}^N]=\E[\beta_{p,k}^N],\
\E[\widehat{\alpha}_{k_1}^N\widehat{\alpha}_{k_2}^N]=\E[\alpha_{k_1}^N\alpha_{k_2}^N].
$$
Under Assumption \ref{assumption:F_Lipschitz}, if $\widehat{\psi}_{\varepsilon,N}(y)$ satisfies the assumptions of Proposition \ref{proposition:bounded_moments} (or Proposition \ref{proposition:bounded_moments_geometric}), for all test function $\phi \in \CC^3_{P}$, there exists $H_0>0$ such that for all $H=N\varepsilon\leq H_0$, the following estimate holds, where $C$ and $K$ are independent of $\varepsilon$ and $N$,
$$\abs{\E[\phi(\varphi_{\varepsilon,T_N}(y))]-\E[\phi(\widehat{\psi}_{\varepsilon,N}(y))]}\leq C(1+\abs{y}^K)H^3,$$
that is, the numerical scheme has weak local order two.
\end{proposition}

\begin{proof}
Using Proposition \ref{proposition:strong_Taylor_development_hitting_times} and its notation $\psi_{\varepsilon,N}(y)$, it is enough to prove that
$$\abs{\E[\phi(\psi_{\varepsilon,N}(y))]-\E[\phi(\widehat{\psi}_{\varepsilon,N}(y))]}\leq C (1+\abs{y}^K) H^3.$$
A local expansion gives 
$$\phi(\psi_{\varepsilon,N}(y))=\phi(y)+\phi'(y)(\psi_{\varepsilon,N}(y)-y)+\phi''(y)(\psi_{\varepsilon,N}(y)-y,\psi_{\varepsilon,N}(y)-y)+R_1.$$
As $\psi_{\varepsilon,N}(y)=\psi_{\varepsilon,T_N}^2(y)$ (see equation \eqref{equation:strong_Taylor_development}), using Inequalities \eqref{equation:Taylor_bound_2}, \eqref{equation:Taylor_bound_3} evaluated at $T_N$ and Proposition \ref{proposition:properties_T1} yield
\begin{align*}
\E[\abs{R_1}]&\leq \E\left[\sup_{x \in [y , \psi_{\varepsilon,N}(y)]} \abs{\phi^{(3)}(x)} \abs{\psi_{\varepsilon,N}(y)-y}^3\right]\\
&\leq \E\left[C(1+\abs{y}^K+\abs{\psi_{\varepsilon,N}(y)}^K) (1+\abs{y}^K) (\varepsilon T_N)^3 e^{C\varepsilon T_N} \right]\\
&\leq \E\left[C (1+\abs{y}^K) (\varepsilon T_N)^3 e^{C\varepsilon T_N} \right]\\
&\leq C (1+\abs{y}^K) H^3.
\end{align*}
We obtain a similar expansion for $\phi(\widehat{\psi}_{\varepsilon,N}(y))$:
$$\phi(\widehat{\psi}_{\varepsilon,N}(y))=\phi(y)+\phi'(y)(\widehat{\psi}_{\varepsilon,N}(y)-y)+\phi''(y)(\widehat{\psi}_{\varepsilon,N}(y)-y,\widehat{\psi}_{\varepsilon,N}(y)-y)+\widehat{R_1},$$
where, using Inequality \eqref{equation:inequality_bounded_moments} (or \eqref{equation:inequality_bounded_moments_geometric}),
\begin{align*}
\E[|\widehat{R_1}|]&\leq \E\left[\sup_{x \in [y , \widehat{\psi}_{\varepsilon,N}(y)]} \abs{\phi^{(3)}(x)} \abs{\widehat{\psi}_{\varepsilon,N}(y)-y}^3\right]\\
&\leq C\E\left[(1+\abs{y}^K+\abs{\widehat{\psi}_{\varepsilon,N}(y)}^K) (1+\abs{y}^K) H^3 M_N^3 \right]\\
&\leq C(1+\abs{y}^K) H^3 \E\left[(1+M_N^K) \right]\\
&\leq C(1+\abs{y}^K) H^3.
\end{align*}
Making the difference of both equations gives 
\begin{align}
\label{equation:proof_local_difference}
\phi(\psi_{\varepsilon,N}(y))-\phi(\widehat{\psi}_{\varepsilon,N}(y))&=\phi'(y)(\psi_{\varepsilon,N}(y)-\widehat{\psi}_{\varepsilon,N}(y))-\phi''(y)(\psi_{\varepsilon,N}(y)-y)^2\\
&+\phi''(y)(\widehat{\psi}_{\varepsilon,N}(y)-y)^2+R,\nonumber
\end{align}
where $\E[\abs{R}]\leq C (1+\abs{y}^K) H^3$.
For the first term of \eqref{equation:proof_local_difference}, we have
\begin{align*}
\E[\phi'(y)(\psi_{\varepsilon,N}(y)-\widehat{\psi}_{\varepsilon,N}(y))]&=
H\sum_{k\in\Z} \E[\phi'(y)(c_k^0(y) (\alpha_k^N-\widehat{\alpha}_k^N))]\\
&+H^2\sum_{p,k\in\Z} \E[\phi'(y)(c_p^1(y)(c_k^0(y)) (\beta_{p,k}^N-\widehat{\beta}_{p,k}^N))].
\end{align*}
Then, we get
$$\E[\phi'(y)(c_k^0(y) (\alpha_k^N-\widehat{\alpha}_k^N))]=\E[\alpha_k^N-\widehat{\alpha}_k^N] \phi'(y)(c_k^0(y))=0.$$
We can do the same thing with the term in $\beta_{p,k}^N$ and obtain 
$$\E[\phi'(y)(\psi_{\varepsilon,N}(y)-\widehat{\psi}_{\varepsilon,N}(y))]=0.$$
Let us now study the second order term $Z=\phi''(y)(\widehat{\psi}_{\varepsilon,N}(y)-y)^2 -\phi''(y)(\psi_{\varepsilon,N}(y)-y)^2$ that appears in \eqref{equation:proof_local_difference}. We develop this expression and keep only the order one and two terms to obtain $Z=H^2 Y+R$ where $\E[\abs{R}]\leq C(1+\abs{y}^K)H^3$ (by the same arguments as before) and
\begin{align*}
Y&=\sum_{k_1,k_2} \left[ \phi''(y)(c_{k_1}^0(y) \widehat{\alpha}_{k_1}^N ,c_{k_2}^0(y) \widehat{\alpha}_{k_2}^N)-\phi''(y)(c_{k_1}^0(y) \alpha_{k_1}^N ,c_{k_2}^0(y) \alpha_{k_2}^N)\right]\\
&=\sum_{k_1,k_2} (\widehat{\alpha}_{k_1}^N \widehat{\alpha}_{k_2}^N-\alpha_{k_1}^N \alpha_{k_2}^N) \phi''(y)(c_{k_1}^0(y),c_{k_2}^0(y))
\end{align*}
The condition on the moments of the $\widehat{\alpha}_k^N$ yields $\E[Y]=0$.

Putting all these arguments together in \eqref{equation:proof_local_difference}, we finally get that
$$\abs{\E[\phi(\psi_{\varepsilon,N}(y))]-\E[\phi(\widehat{\psi}_{\varepsilon,N}(y))]}\leq C(1+\abs{y}^K)H^3.$$
We deduce the local order two of the proposed numerical scheme.
\end{proof}

\begin{remark*}
The constant $H_0$ in Proposition \ref{proposition:local_weak_error} depends on $F$, but also of the polynomial growth power of $\phi$ and its first three derivatives.
This dependence is expected when trying to evaluate the solution of SDEs at random times.
To make $H_0$ independent of the test functions, one can consider the following sets of test functions
$$\CC^3_{P,K}=\{\phi \in \CC^3, \exists C>0,\exists k\leq K,\forall y, \abs{\phi^{(i)}(y)}\leq C(1+\abs{y}^k), i\in \{0,1,2,3\}\}.$$
\end{remark*}

\subsection{Global error}
\label{section:global_error_proof}

\begin{theorem}[Global convergence]
\label{theorem:global_convergence_theorem}
Assume that the numerical scheme $\widehat{\psi}_{\varepsilon,N}$ satisfies equation \eqref{equation:proof_integrator_order_2} (respectively equation \eqref{equation:proof_geometric_integrator_order_2})
where $\widehat{\alpha}_k^N\in\C$, $\widehat{\beta}_{p,k}^N\in \R$ (respectively $\widehat{\widetilde{\beta}}_{p,k}^N\in \R$) are random variables such that $\widehat{\alpha}_k^N=\overline{\widehat{\alpha}_{-k}^N}$ and
$$
\E[\widehat{\alpha}_k^N]=\E[\alpha_k^N],\
\E[\widehat{\beta}_{p,k}^N]=\E[\beta_{p,k}^N],\
\E[\widehat{\alpha}_{k_1}^N\widehat{\alpha}_{k_2}^N]=\E[\alpha_{k_1}^N\alpha_{k_2}^N].
$$
(respectively $\widehat{\alpha}_k^N$ satisfies the same conditions and $\widehat{\widetilde{\beta}}_{p,k}^N$ satisfies $\E[\widehat{\widetilde{\beta}}_{p,k}^N]=\E[\widetilde{\beta}_{p,k}^N]$).
Under Assumption \ref{assumption:F_Lipschitz}, if for all $q>0$, $\E\left[\left(\sum_k \frac{\abs{\widehat{\alpha}_k^N}^2}{k^2}\right)^{q}\right]$ and $\E\left[\left(\sum_{p,k} \frac{\abs{\widehat{\beta}_{p,k}^N}^2}{k^2}\right)^{q}\right]$ are bounded uniformly in $N$
(respectively $\sum_{k} \abs{\widehat{\alpha}_k^N}$, $\sum_{p,k} \abs{\widehat{\widetilde{\beta}}_{p,k}^N}$, $\E\left[\left(\sum_k \frac{\abs{\widehat{\alpha}_k^N}^2}{k^2}\right)^{q}\right]$ and $\E\bigg[\bigg(\sum_{p,k} \frac{\abs{\widehat{\widetilde{\beta}}_{p,k}^N}^2}{k^2}\bigg)^{q}\bigg]$ are bounded uniformly in $N$), for all $T>0$, for all test function $\phi \in \CC^3_P$, there exists $H_0>0$ such that for all $H\leq H_0$, for all $M\geq 0$ such that $MN\varepsilon=MH\leq T$, there exists two positive constants $K$ and $C$ both independent of $\varepsilon$ and $N$ such that
$$\abs{\E[\phi(\varphi_{\varepsilon,T_{NM}}(X_0))]-\E[\phi(\widehat{\psi}_{\varepsilon,N}^M(X_0))]}\leq CH^2(1+\E[\abs{X_0}^{K}]).$$
\end{theorem}

\begin{proof}
We denote $$e_M=\E[\phi(\varphi_{\varepsilon,T_{NM}}(X_0))]-\E[\phi(\widehat{\psi}_{\varepsilon,N}^{M}(X_0))]$$ and rewrite it with a telescopic sum
\begin{align*}
e_M&=\sum_{m=1}^M \E[\phi(\varphi_{\varepsilon,T_{N(m-1)}}(\widehat{\psi}_{\varepsilon,N}^{M-m+1}(X_0)))]-\E[\phi(\varphi_{\varepsilon,T_{Nm}}(\widehat{\psi}_{\varepsilon,N}^{M-m}(X_0)))]\\
&=\sum_{m=1}^M \E[\widetilde{\phi}_{m-1}(\widehat{\psi}_{\varepsilon,N}(\widehat{\psi}_{\varepsilon,N}^{M-m}(X_0)))]-\E[\widetilde{\phi}_{m-1}(\varphi_{\varepsilon,T_N}(\widehat{\psi}_{\varepsilon,N}^{M-m}(X_0)))]
\end{align*}
where $\widetilde{\phi}_{m-1}=\phi \circ \varphi_{\varepsilon,T_{N(m-1)}}$.
Using Lemma \ref{lemma:regularity_varphi} and $\phi \in \CC^3_P$, we obtain for $0\leq i\leq 3$,
$$\abs{\widetilde{\phi}_m^{(i)}(y)}\leq Ce^{C\varepsilon T_{Nm}}(1+\abs{y}^K).$$
Thus, knowing the hitting times involved, $\widetilde{\phi}_m\in \CC^3_P$.
Using Assumption \ref{assumption:F_Lipschitz}, $(c_p^0)'=c_p^1$ and $\beta_{p,k}^N=\widetilde{\beta}_{p,k}^N+\frac{\alpha_p \alpha_k}{2}$, we deduce that $\widehat{\psi}_{\varepsilon,N}$ satisfies the assumptions of Proposition \ref{proposition:local_weak_error}. Applying Proposition \ref{proposition:local_weak_error} to each term of $e_M$ gives
$$\abs{e_M}\leq \sum_{m=1}^M C\E\left[e^{C\varepsilon T_{Nm}}\right]H^3\left(1+\E\left[\abs{\widehat{\psi}_{\varepsilon,N}^{M-m}(X_0)}^K\right]\right).$$
Finally, the moments of $\widehat{\psi}_{\varepsilon,N}^{m}(X_0)$ are all bounded uniformly in $\varepsilon$, $N$ and $m$ according to Proposition \ref{proposition:bounded_moments} (respectively \ref{proposition:bounded_moments_geometric}).
Thus
$$\abs{e_M}\leq \sum_{m=1}^M CH^3(1+\E[\abs{X_0}^{K}]) \leq CH^2(1+\E[\abs{X_0}^{K}]).$$
We deduce the global weak order two.
\end{proof}

With the help of Theorem \ref{theorem:global_convergence_theorem}, we prove Proposition \ref{proposition:asymptotic_model} and the convergence of Methods \ref{algorithm:weak_order_2} and \ref{algorithm:geometric_weak_order_2}.

\begin{proof}[Proof of Proposition \ref{proposition:asymptotic_model}]
Rewriting Theorem \ref{theorem:global_convergence_theorem} for order one yields for all $H=N\varepsilon$ small enough and all $M\geq 0$,
$$\abs{\E[\phi(\varphi_{\varepsilon,T_{NM}}(X_0))]-\E[\phi(y_{\varepsilon NM})]}\leq C(\varepsilon N)^2(1+\E[\abs{X_0}^{K}]).$$
Evaluating in $N=1$, $M=\frac{T}{\varepsilon}$ and taking the limit $\varepsilon\rightarrow 0$ yield the result.
\end{proof}

\begin{proof}[Proof of Theorem \ref{theorem:convergence_algorithm_order_2}]
As $\widehat{\alpha}_k^N\leq C$ and $\sum_{p,k} \frac{\abs{\E[\beta_{p,k}^N]}^2}{k^2}$ converges by Proposition \ref{proposition:moments_alpha_beta}, Theorem \ref{theorem:global_convergence_theorem} applies and concludes the proof.
\end{proof}

\begin{proof}[Proof of Theorem \ref{theorem:convergence_algorithm_order_2_geometric}]
The regularity assumptions yield the Lipschitzness of the $c_k^0(y)$ and the involved $c_p^1(y)(c_k^0(y))$ with constants independent of $k$ and $p$. As $\sum_{k} \abs{\widehat{\alpha}_k^N}$ and $\sum_{p,k} \abs{\widehat{\widetilde{\beta}}_{p,k}^N}$ are bounded, the right hand-side of equation \eqref{equation:geometric_order_2_scheme} is a contraction for all $H\leq H_0$ small enough and the constant does not depend on $Y_m$, so $H_0$ depends only of $F$ and $F'$. Thus, the integrator is well-posed for all $H\leq H_0$.

The weak order two is obtained using Theorem \ref{theorem:global_convergence_theorem}. Indeed the use of discrete random variables and Proposition \ref{proposition:moments_alpha_beta} give the convergence of the involved series.

For showing that Method \ref{algorithm:geometric_weak_order_2} preserves quadratic invariants, it is sufficient to prove that $\invariant'(y)(\sum_k c_k^0(y)\widehat{\alpha}_k^N)=0$ and $\invariant'(y)(\sum_{p,k} c_p^1(y)(c_k^0(y)) \widehat{\widetilde{\beta}}_{p,k}^N)=0$ (see \cite[Chap. IV]{Hairer06gni}).
The preservation of $\invariant$ by equation \eqref{equation:NLS} yields $\invariant'(y)(Ay)=0$ and $\invariant'(y)(F(y))=0$. We deduce the following two equations, valid for all $y\in \R^d$,
\begin{align}
\label{equation:lemma_g_proof_preservation1}
y^T \invariantmatrix g_\theta^0(y) &=0,\\
\label{equation:lemma_g_proof_preservation2}
y^T \invariantmatrix g_\theta^1(y)(g_\nu^0(y)) &=-(g_\nu^0(y))^T \invariantmatrix g_\theta^0(y),
\end{align}
where equation \eqref{equation:lemma_g_proof_preservation2} is obtained by differentiating equation \eqref{equation:lemma_g_proof_preservation1} in the direction $g_\nu^0$.
Using equation \eqref{equation:lemma_g_proof_preservation1}, we have
$$\invariant'(y)\left(\sum_k c_k^0(y)\widehat{\alpha}_k^N\right)=\int_0^1 \invariant'(y)(g_\theta^0(y)) \sum_k e^{-2i\pi k\theta}\widehat{\alpha}_k^N d\theta =0.$$
\pagebreak[2]
For the second order term, equation \eqref{equation:lemma_g_proof_preservation2} and the values of Proposition \ref{proposition:moments_alpha_beta} yield
\begin{align*}
\invariant'(y)&\left(\sum_{p,k} c_p^1(y)(c_k^0(y)) \widehat{\widetilde{\beta}}_{p,k}^N\right)\\
&=\int_0^1\int_0^1 y^T \invariantmatrix g_\theta^1(y)(g_\nu^0(y))\sum_{p,k} e^{-2i\pi p\theta}e^{-2i\pi k\nu} \widehat{\widetilde{\beta}}_{p,k}^N d\nu d\theta\\
&=-\int_0^1\int_0^1 (g_\nu^0(y))^T \invariantmatrix g_\theta^0(y)\sum_{p,k} e^{-2i\pi p\theta}e^{-2i\pi k\nu} \widehat{\widetilde{\beta}}_{p,k}^N d\nu d\theta\\
&=-\frac{1}{2}\int_0^1\int_0^1 (g_\nu^0(y))^T \invariantmatrix g_\theta^0(y)\sum_{p,k} [e^{-2i\pi p\theta}e^{-2i\pi k\nu} +e^{-2i\pi p\nu}e^{-2i\pi k\theta}]\widehat{\widetilde{\beta}}_{p,k}^N d\nu d\theta\\
&=0.
\end{align*}
Hence Method B is well-posed, has weak order 2 and preserves the invariant $\invariant$.
%
\end{proof}

\section{Numerical experiments}
\label{section:numerical experiments}

In this section, we first illustrate numerically the weak order two of Methods \ref{algorithm:weak_order_2} and \ref{algorithm:geometric_weak_order_2} with convergence curves.
Then, we apply the new algorithms to solve the nonlinear Schrödinger equation with highly-oscillatory white noise dispersion \eqref{equation:NLS_WND_general}.

\subsection{Weak order of convergence}

To confirm the results of Theorem \ref{theorem:convergence_algorithm_order_2} and Theorem \ref{theorem:convergence_algorithm_order_2_geometric}, we check numerically if Methods \ref{algorithm:weak_order_2} and \ref{algorithm:geometric_weak_order_2} have weak order two of accuracy w.r.t.\ts $H$ uniformly in $\varepsilon$ and $N$. 
As the Euler-Maruyama method and the algorithms presented in \cite{Cohen12otn,Belaouar15nao,Cohen17eif} are completely innacurate if they do not satisfy the severe timestep restriction $h\ll \varepsilon$, we compare the performance of Methods \ref{algorithm:weak_order_2} and \ref{algorithm:geometric_weak_order_2} to the performance of the Euler method \eqref{algorithm:weak_order_1}.
We first apply the algorithms on equation \eqref{equation:NLS} with the linearity $F(y)=iy$, $A=2i\pi$, $X_0=1$ and $\varepsilon=10^{-3}$. Equivalently we can write it in the real setting as
$$dX=\frac{2\pi}{\sqrt{\varepsilon}} \begin{pmatrix}
0&-1\\1&0
\end{pmatrix} X\circ dW +\begin{pmatrix}
0&-1\\1&0
\end{pmatrix} X dt,\ X_0=\begin{pmatrix}
1\\0
\end{pmatrix}.$$
We plot on a logarithmic scale an estimate of the weak error for approximating $X$ at time $T=10^{-3} T_{2^8}$ where $\E[T]=0.256$.
The exact solution $X(T)$ is approximated by the output of Method \ref{algorithm:geometric_weak_order_2} for $H=\varepsilon$. The parameters $N$ and $m$ are varying under the condition that $Nm=2^8$. The test function is $\phi(y)=2y_1+4y_2$ and the average is taken over $10^7$ trajectories. We choose the tolerance $10^{-13}$ for the fixed point.
On the right picture of Figure \ref{figure:weak_convergence_curves}, we use a modification of a Kubo oscillator introduced in \cite{Cohen12otn} with the nonlinearity $F(y)=i(1+\Real(y)^3+\Imag(y)^5)y$. In the real setting, it yields the following two-dimensional SDE
$$dX=\frac{2\pi}{\sqrt{\varepsilon}} \begin{pmatrix}
0&-1\\1&0
\end{pmatrix} X\circ dW +\begin{pmatrix}
0&-1\\1&0
\end{pmatrix} (1+X_1^3+X_2^5) X dt,\ X_0=\begin{pmatrix}
1\\0
\end{pmatrix}.$$
We take 8 modes for the Fourier decomposition and the same other parameters as before. The average is taken over $10^6$ trajectories.

\begin{figure}[t]
	\begin{minipage}[c]{.49\linewidth}
		\includegraphics[scale=0.5]{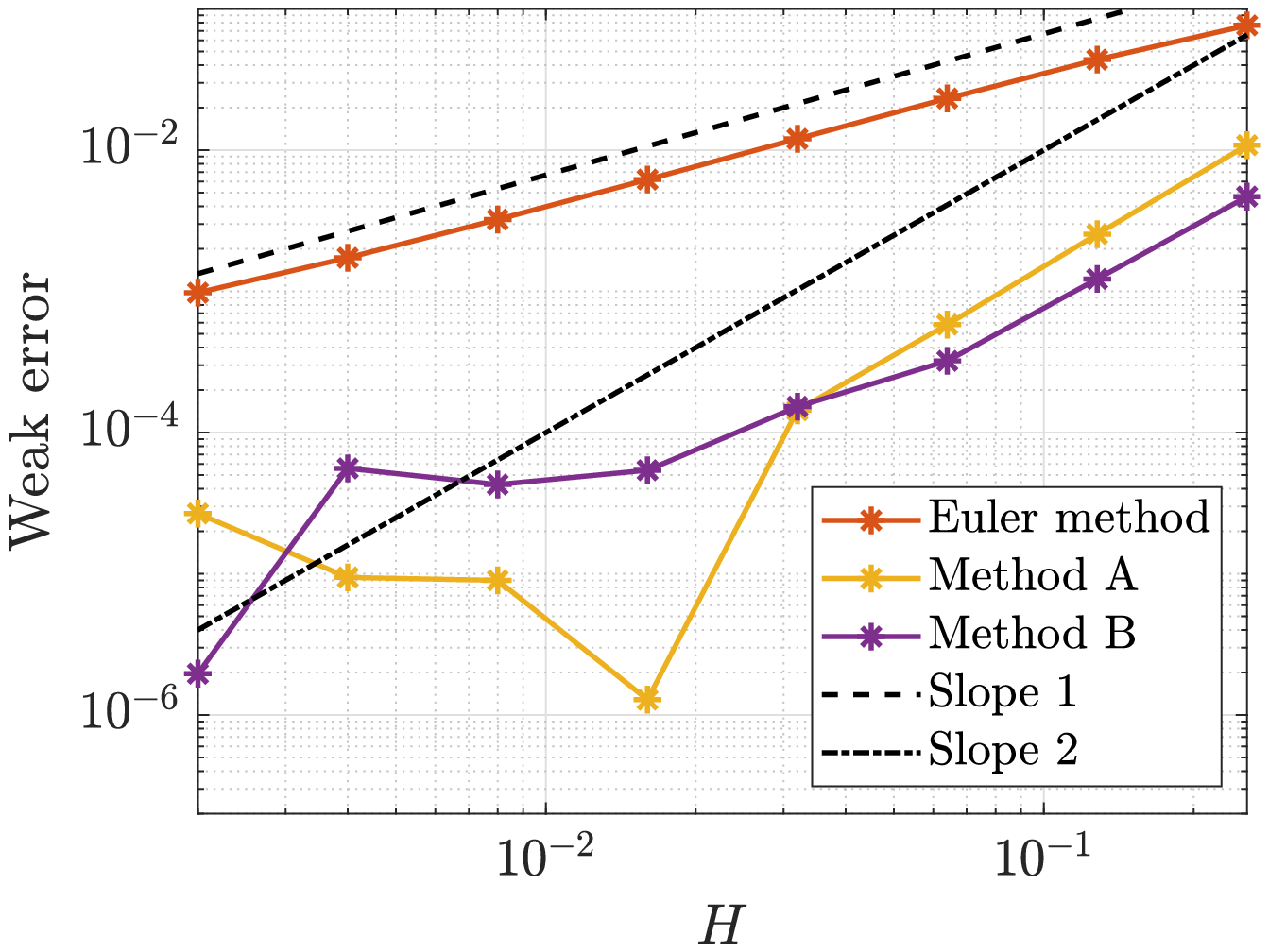}
	\end{minipage} \hfill
	\begin{minipage}[c]{.49\linewidth}
		\includegraphics[scale=0.5]{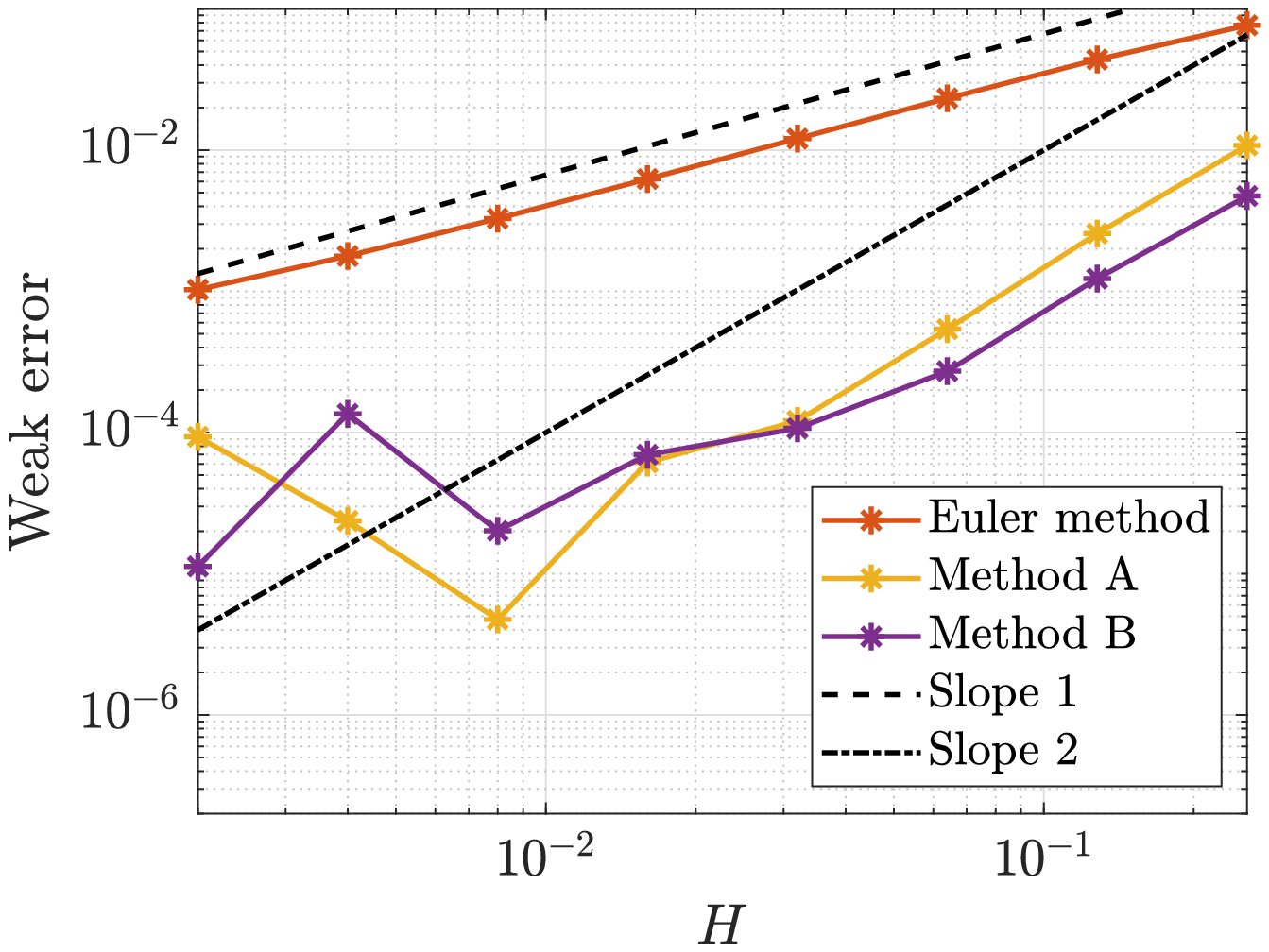}
	\end{minipage}
	\caption{Weak error versus the stepsize $H=N\varepsilon$ for approximating the solution of equation \eqref{equation:NLS} at time $\varepsilon T_{2^8}$ for the linear $F(y)=iy$ (left) and the non-linear $F(y)=i(1+\Real(y)^3+\Imag(y)^5)y$ (right) with $A=2i\pi$, $X_0=1$, $\varepsilon=10^{-3}$ and the test function $\phi(y)=2\Real(y)+4\Imag(y)$.}
	\label{figure:weak_convergence_curves}
\end{figure}

In both cases, we observe the weak order two of Methods \ref{algorithm:weak_order_2} and \ref{algorithm:geometric_weak_order_2}. The irregularities of the curve for a small $H$ come from Monte-Carlo errors.
We repeated the same experiment on many other examples and we always observe the desired order two as long as $H$ is small enough.

\subsection{Numerical experiments on NLS equation with white noise dispersion}

We now apply the algorithms to solve on the torus $\T=[-\pi,\pi]$ the following SPDE of the form \eqref{equation:NLS_WND_general}, with a polynomial linearity and the stiffness parameter $\varepsilon=10^{-2}$,
\begin{equation}
\label{equation:polynomial_NLS_white_noise_dispersion}
du=\frac{2\pi}{\sqrt{\varepsilon}}\begin{pmatrix}
0&-1\\1&0
\end{pmatrix} \Delta u\circ dW+\begin{pmatrix}
0&-1\\1&0
\end{pmatrix}\abs{u}^{2\sigma}u dt,\ x\in \T,\ t>0,
\end{equation}
where the unknown $u$ is a random process depending on $x\in \T$ and $t\geq 0$.
We consider a spectral discretization in space of this equation with $K_x=2^7$ modes $u(x,t)\approx \sum_{\abs{l}\leq K_x} Y_l(t) e^{i l x}$.
We obtain an equation of the desired form \eqref{equation:NLS} with a truncated nonlinearity and the block-diagonal matrix 
$$A=\Diag(-2\pi l^2 \begin{pmatrix}
0&-1\\1&0
\end{pmatrix}, \abs{l}\leq K_x).$$
Beginning with the initial condition $u_0(x)=\exp(-3x^4+x^2)$ on $\T$ that decreases fast enough, we apply Methods \ref{algorithm:weak_order_2} and \ref{algorithm:geometric_weak_order_2} in the two cases $\sigma=2$ and $\sigma=4$ with $K_t =2^6$ modes, $N=10$ revolutions, $m=150$ iterations and a tolerance of $10^{-13}$ for the fixed point iteration.
Figure \ref{figure:Plot_NLS_WND} shows the evolution in time of one trajectory given by Method \ref{algorithm:geometric_weak_order_2} (with a 300 points evaluation grid in space).

\begin{figure}[t]
	\begin{minipage}[c]{.49\linewidth}
		\includegraphics[scale=0.5]{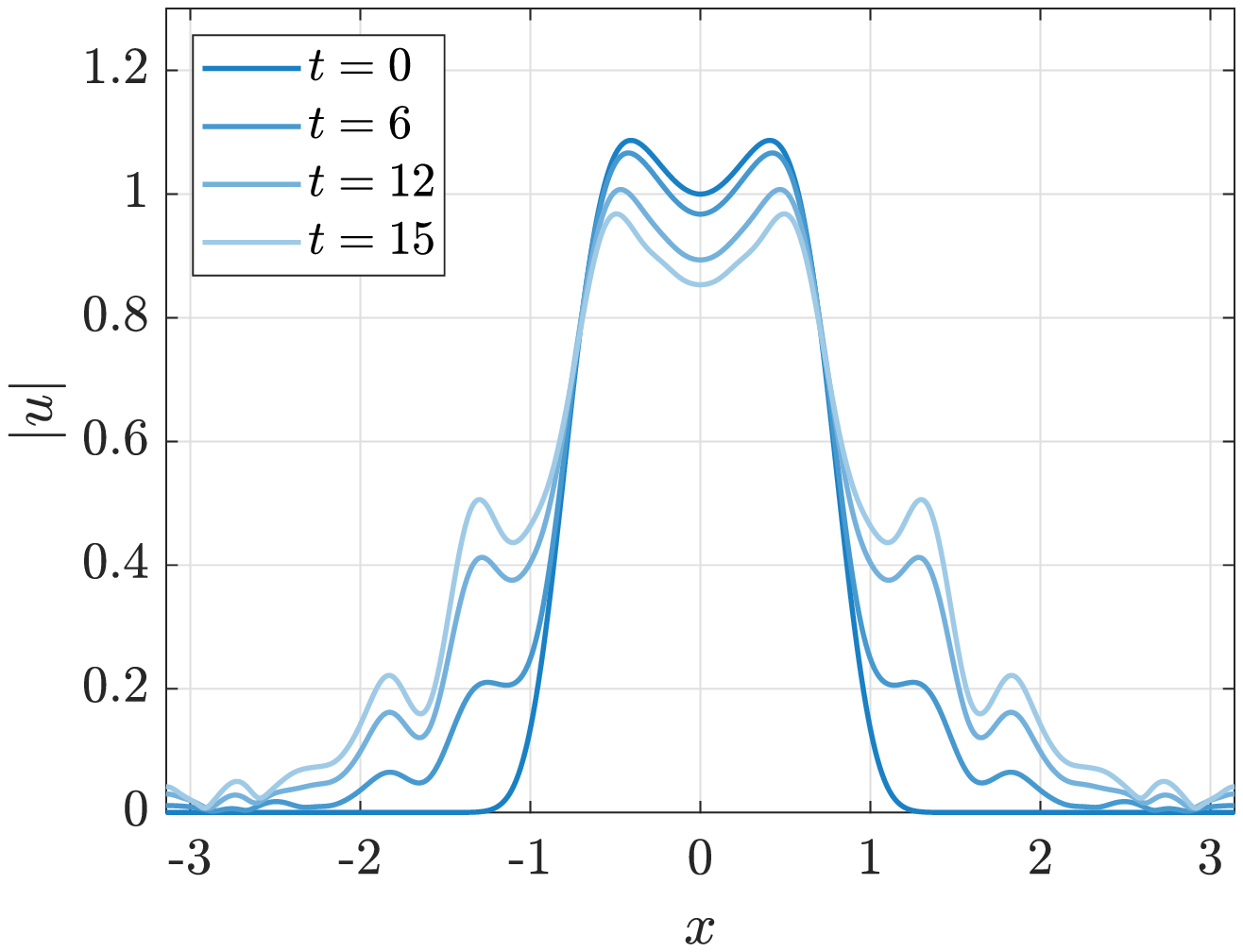}
	\end{minipage} \hfill
	\begin{minipage}[c]{.49\linewidth}
		\includegraphics[scale=0.5]{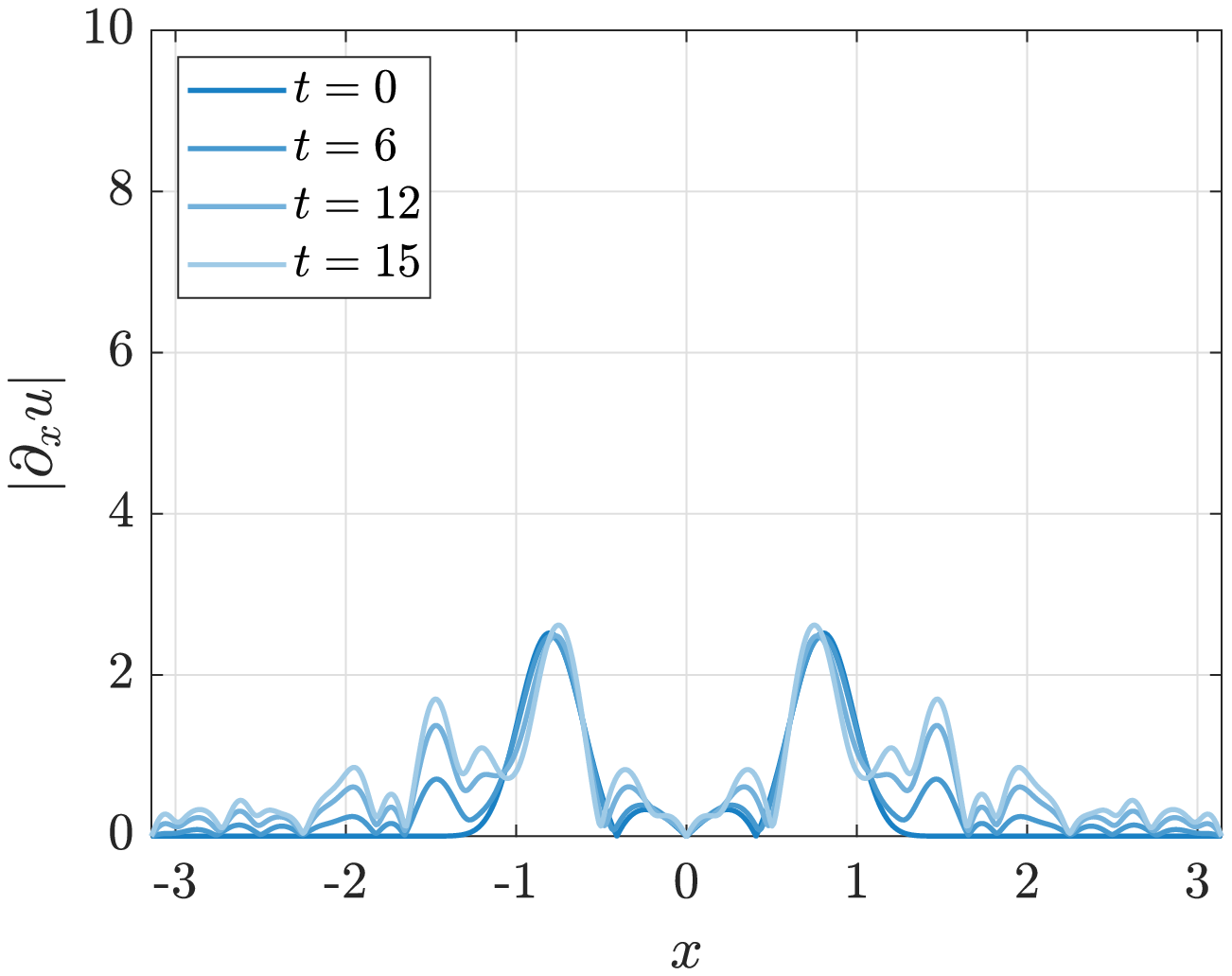}
	\end{minipage}
\vskip\baselineskip
	\begin{minipage}[c]{.49\linewidth}
		\includegraphics[scale=0.5]{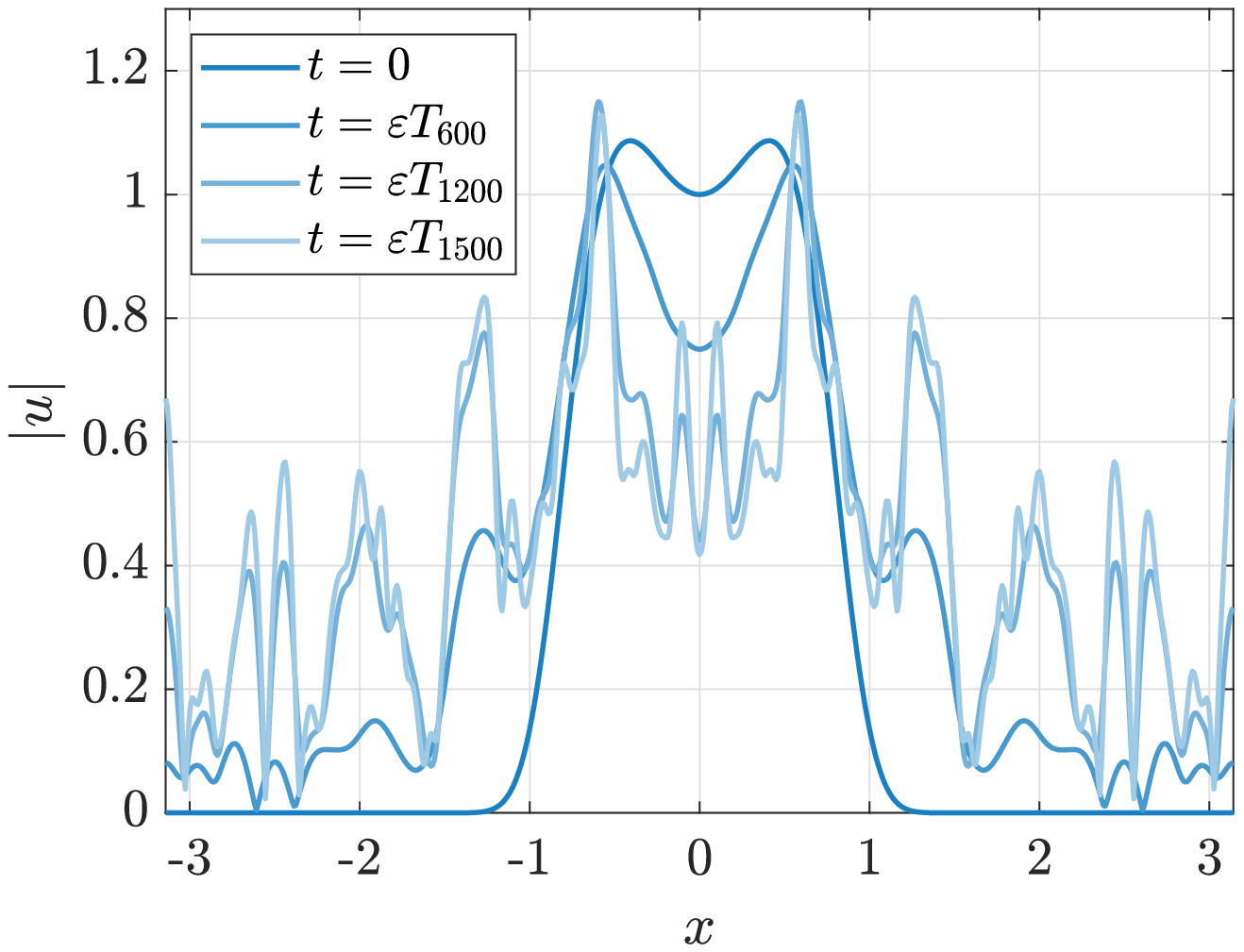}
	\end{minipage} \hfill
	\begin{minipage}[c]{.49\linewidth}
		\includegraphics[scale=0.5]{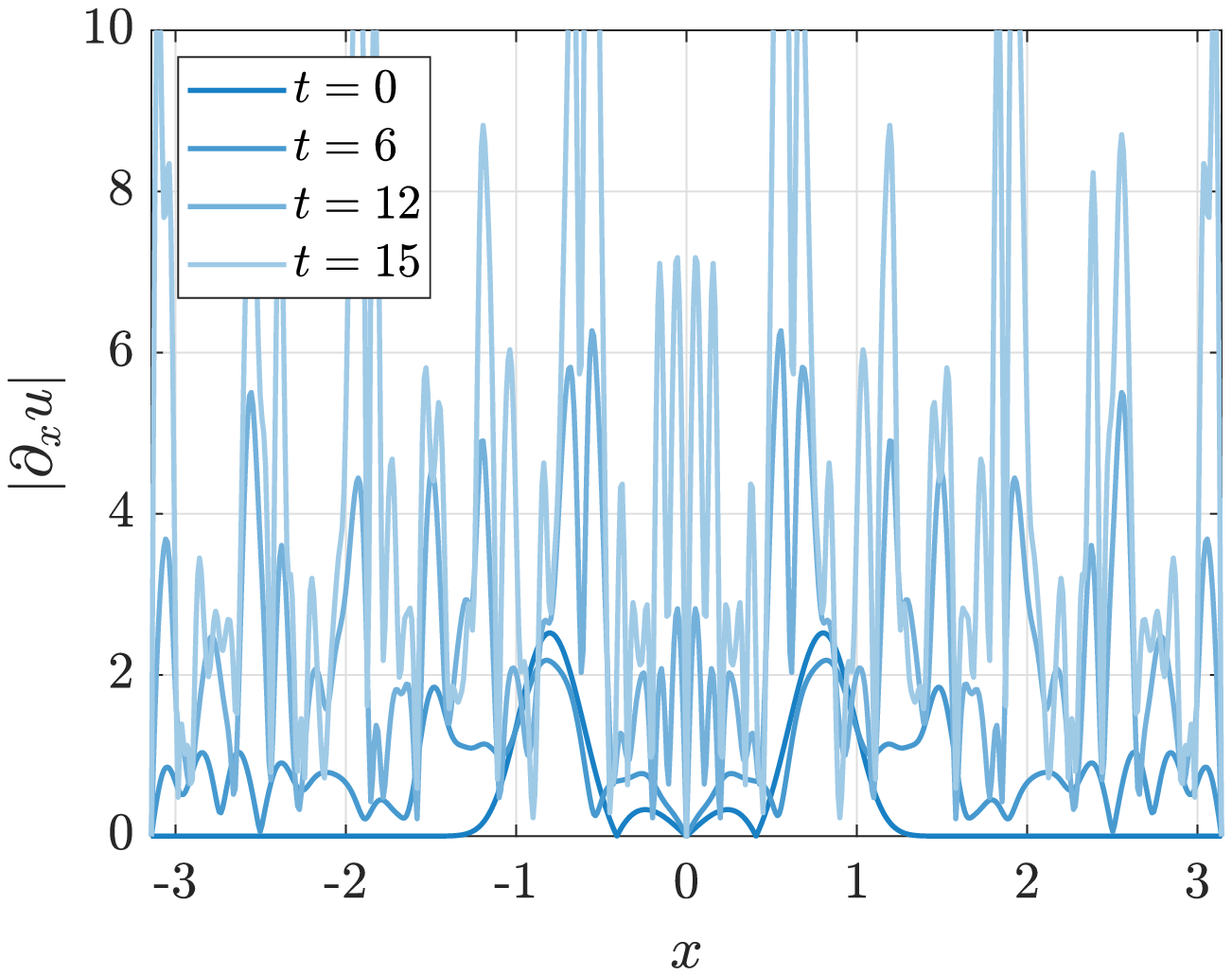}
	\end{minipage}
	\caption{Approximation by Method \ref{algorithm:geometric_weak_order_2} of $\abs{u}$ and $\abs{\partial_x u}$ with $u$ solution of a spatial discretization with $K_x=2^7$ modes of the nonlinear Schrödinger equation with white noise dispersion \eqref{equation:polynomial_NLS_white_noise_dispersion} on the torus $\T=[-\pi,\pi]$ with the parameters $\varepsilon=10^{-2}$, $\sigma=2$ (top) and $\sigma=4$ (bottom).}
	\label{figure:Plot_NLS_WND}
\end{figure}

In Figure \ref{figure:Norms_NLS_WND}, we observe the discrete $L^2$ and $H^1$ norms behaviour of one trajectory given by our two algorithms and the Euler method \eqref{algorithm:weak_order_1} (the simulated $(\alpha_k)_k$ are the same for Methods \ref{algorithm:weak_order_2} and \ref{algorithm:geometric_weak_order_2}).
The Euler method quickly blows up in both norms. The $L^2$ norm of Method \ref{algorithm:weak_order_2} is not conserved. In contrast, Method \ref{algorithm:geometric_weak_order_2} preserves the $L^2$ norm according to Theorem \ref{theorem:convergence_algorithm_order_2_geometric}.
When $\sigma=4$, numerical simulations hint that a blow-up in the $H^1$ norm always happens for all considered methods at a certain time that increases as $\varepsilon$ goes to zero.
We recall that in the optic fiber model \eqref{equation:NLS_WND_general}, $t$ represents the distance along the optic fiber and a cubic nonlinearity ($\sigma=2$) is typically considered \cite{Garnier02sod}.
For $\sigma=2$, we do not observe any blow-up in the $H^1$ norm in Figure \ref{figure:Norms_NLS_WND}, suggesting the well-posedness of the model for all optic fiber distance.
Also, the larger $\sigma$ is, the sooner the blow-up happens.
These behaviors agree with the blow-up conjecture for $\varepsilon=1$ and $\sigma\geq 4$ presented in \cite{Belaouar15nao}, and suggest that the conjecture persists in the highly-oscillatory regime $\varepsilon \ll 1$.
\begin{figure}[t]
	\begin{minipage}[c]{.49\linewidth}
		\includegraphics[scale=0.5]{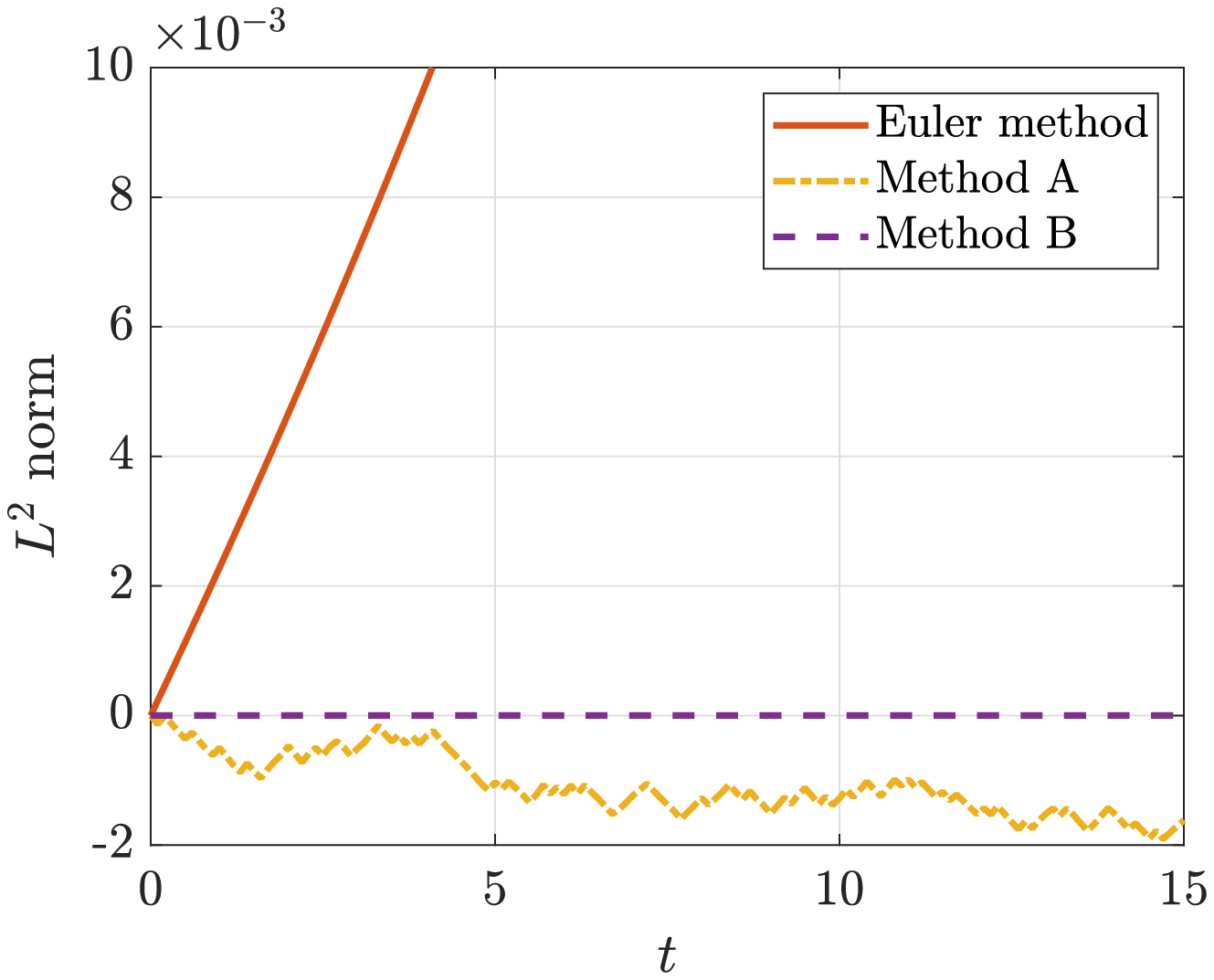}
	\end{minipage} \hfill
	\begin{minipage}[c]{.49\linewidth}
		\includegraphics[scale=0.5]{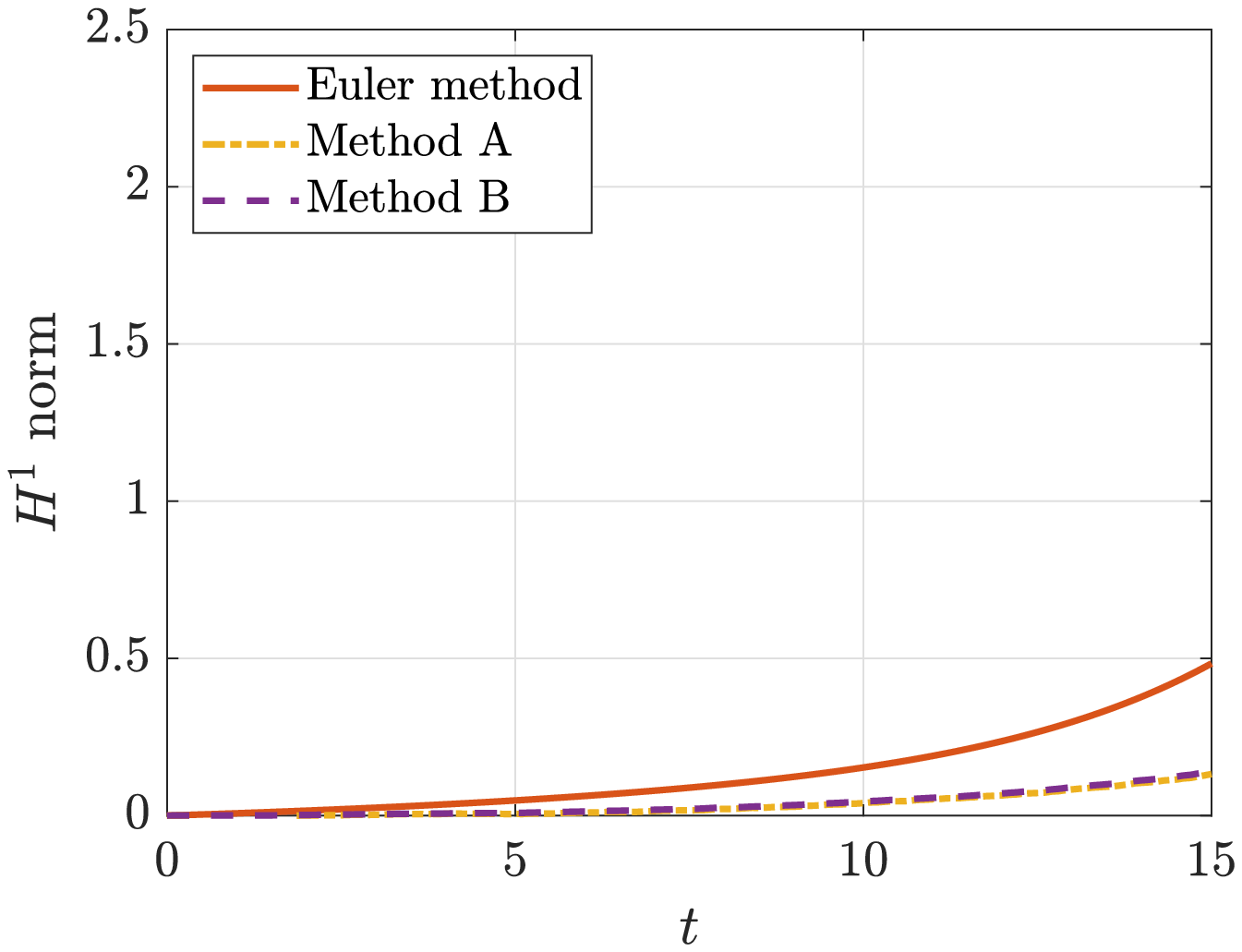}
	\end{minipage}
\vskip\baselineskip
	\begin{minipage}[c]{.49\linewidth}
		\includegraphics[scale=0.5]{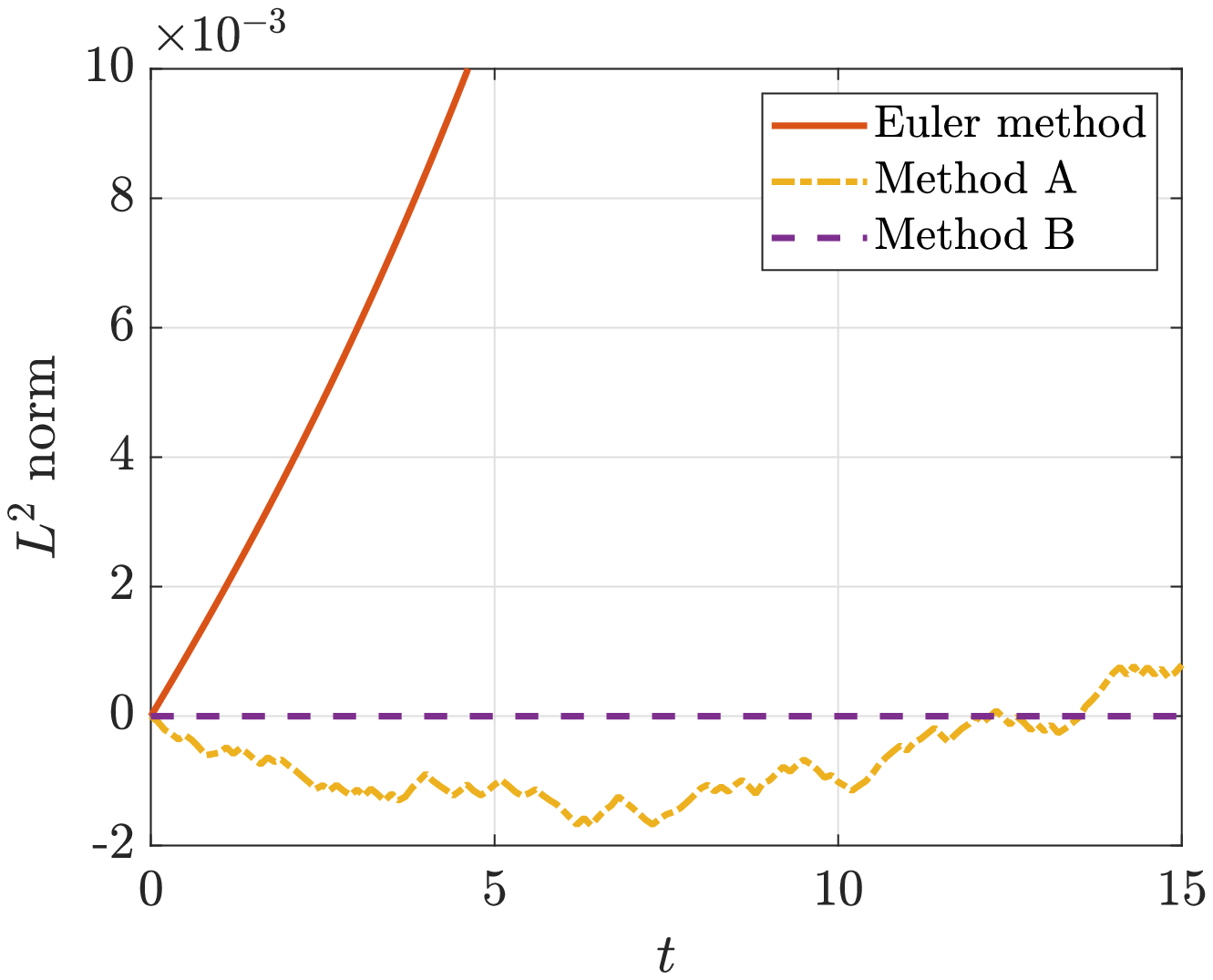}
	\end{minipage} \hfill
	\begin{minipage}[c]{.49\linewidth}
		\includegraphics[scale=0.5]{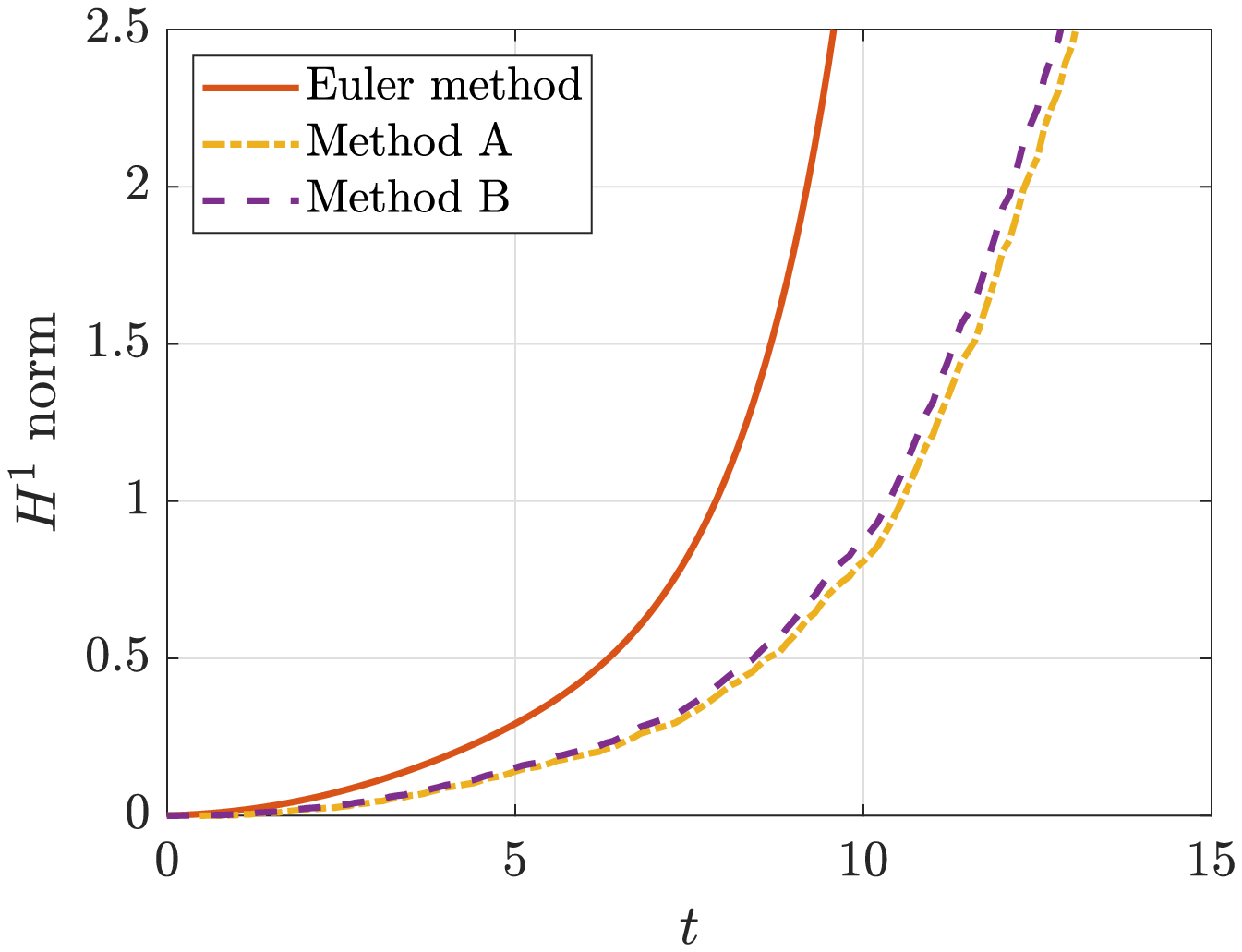}
	\end{minipage}
	\caption{Evolution in long time of the quantities $\norme{U_t}_{L^2}-\norme{u_0}_{L^2}$ (left) and $\norme{U_t}_{H^1}-\norme{u_0}_{H^1}$ (right) with $U_t$ the approximation computed with Euler method and Methods \ref{algorithm:weak_order_2} and \ref{algorithm:geometric_weak_order_2} for one trajectory of equation \eqref{equation:polynomial_NLS_white_noise_dispersion} with $\varepsilon=10^{-2}$, $\sigma=2$ (top) and $\sigma=4$ (bottom).}
	\label{figure:Norms_NLS_WND}
\end{figure}

\bigskip

\section*{Acknowledgments}
The authors would like to thank Georg Gottwald for helpful and stimulating discussions.
This work was partially supported by the Swiss National Science Foundation, grants No. 200020\_184614, No. 200021\_162404 and No. 200020\_178752.
The computations were performed at the University of Geneva on the Baobab cluster using the Julia programming language.

\bibliographystyle{abbrv}
\bibliography{Ma_Bibliographie}

\appendix
\section*{Appendix} 



\begin{proof}[Proof of Lemma \ref{lemma:regularity_varphi}]
\begin{enumerate}
\item First, $\varphi_{\varepsilon,t}(y)$ is the solution of
$$\varphi_{\varepsilon,t}(y)=e^{AW(t)}y+\varepsilon e^{AW(t)} \int_0^{t} e^{-AW(s)} F(\varphi_{\varepsilon,s}(y)) ds.$$
Using the boundedness of the continuous periodic function $\theta\rightarrow e^{\theta A}$ and Assumption \ref{assumption:F_Lipschitz}, we get
$$\abs{\varphi_{\varepsilon,t}(y_1)-\varphi_{\varepsilon,t}(y_2)}\leq \abs{y}+L\varepsilon \int_0^t \abs{\varphi_{\varepsilon,s}(y_1)-\varphi_{\varepsilon,s}(y_2)} ds.$$
The Gronwall lemma yields the desired bound.
\item Straightforward using previous statement.
\item Differentiating the integral formulation defining $\varphi_{\varepsilon,t}(y)$ gives
$$\partial_y \varphi_{\varepsilon,t}(y)(h)=e^{AW(t)}h+\varepsilon e^{AW(t)} \int_0^{t} e^{-AW(s)} F'(\varphi_{\varepsilon,s}(y))(\partial_y \varphi_{\varepsilon,s}(y)(h)) ds.$$
Then Assumption \ref{assumption:F_Lipschitz} yields
$$\abs{\partial_y \varphi_{\varepsilon,t}(y)(h)}\leq \abs{h}+L \varepsilon \int_0^{t} \abs{\partial_y \varphi_{\varepsilon,t}(y)(h)} ds.$$
The Gronwall lemma allows to obtain
$$\abs{\partial_y \varphi_{\varepsilon,t}(y)}\leq e^{L\varepsilon t}.$$
For the second derivative, we get
\begin{align*}
\partial_y^2 \varphi_{\varepsilon,t}(y)(h,k)&=\varepsilon e^{AW(t)} \int_0^{t} e^{-AW(s)} [F'(\varphi_{\varepsilon,s}(y))(\partial_y^2 \varphi_{\varepsilon,s}(y)(h,k))\\
&+F''(\varphi_{\varepsilon,s}(y))(\partial_y \varphi_{\varepsilon,s}(y)(h),\partial_y \varphi_{\varepsilon,s}(y)(k))] ds.
\end{align*}
Then
\begin{align*}
\abs{\partial_y^2 \varphi_{\varepsilon,t}(y)(h,k)}&\leq
C\varepsilon \int_0^{t} [(1+\abs{\varphi_{\varepsilon,s}(y)}^K)\abs{\partial_y \varphi_{\varepsilon,s}(y)(h)}\abs{\partial_y \varphi_{\varepsilon,s}(y)(k)}\\
&+\abs{\partial_y^2 \varphi_{\varepsilon,s}(y)(h,k)}] ds\\
&\leq C\varepsilon t (1+\abs{y}^K)e^{C\varepsilon t}\abs{h}\abs{k} + C\varepsilon \int_0^{t} \abs{\partial_y^2 \varphi_{\varepsilon,s}(y)(h,k)}.
\end{align*}
Then the Gronwall lemma allows to conclude.
The proof is similar for the third derivative.
\end{enumerate}
\end{proof}

\end{document}